\documentclass[reqno]{amsart}
\usepackage[colorlinks=true]{hyperref}
\usepackage{amssymb,amsfonts,amsmath,amsthm,mathrsfs}
\usepackage{cite} 
\usepackage{amsaddr}   
    
\title[Invariant manifolds for parabolic equations]
      {Invariant manifolds for parabolic equations under perturbation of the domain}
\author[P. Sa Ngiamsunthorn]{Parinya Sa Ngiamsunthorn}
\address{School of Mathematics and Statistics\\
  University of Sydney, NSW 2006  Australia}
\email{pasa4391@uni.sydney.edu.au}
   
\theoremstyle{plain}
\newtheorem{theorem}{Theorem}[section]

\newtheorem{corollary}[theorem]{Corollary}
\newtheorem{lemma}[theorem]{Lemma}

\theoremstyle{definition}
\newtheorem{definition}[theorem]{Definition}

\theoremstyle{remark}
\newtheorem{remark}[theorem]{Remark}

\newtheorem{assumption}[theorem]{Assumption}

\DeclareMathOperator{\graph}{graph}
\DeclareMathOperator{\real}{Re}
\begin{document}
\date{August 29, 2011}

\begin{abstract}
We study the effect of domain perturbation on invariant manifolds for  
semilinear parabolic equations subject to Dirichlet boundary 
condition. Under Mosco convergence assumption on the domains, 
we prove the upper and lower semicontinuity of 
both the local unstable invariant manifold and the local stable invariant manifold
near a hyperbolic equilibrium.  
The continuity results are obtained by keeping track of the construction of invariant manifolds 
in P. W. Bates and C. K. R. T. Jones [Dynam. Report. Ser. Dynam. Systems Appl. Vol. 2, 1--38, 1989].
%rather than a standard construction in D. Henry [Lecture Notes in Math. Vol. 840, 1981].
\end{abstract}
%typeset title and author information
\maketitle

\section{Introduction}
The study of invariant manifolds is an important tool to understand the behaviour of a dynamical system
near an equilibrium point. In this paper, we are interested in dynamical systems arising from
semilinear parabolic equations.
Let $\Omega$ be a bounded open set in $\mathbb R^N$, $N \geq 2$. We consider
the parabolic equation of the form
\begin{equation}
\label{eq:paraInvMan}
  \left\{ 
  \begin{aligned}
    \frac{\partial u}{\partial t} + \mathcal A u  &=g(x,u) &&\quad \text { in } \Omega \times (0,\infty)\\
    u &= 0  &&\quad \text{ on } \partial \Omega \times (0,\infty)\\
     u(\cdot,0) &= u_0 &&\quad \text{ in } \Omega , \\
  \end{aligned} 
  \right . 
\end{equation}
where $g$ is a function in $C(\mathbb R^N \times \mathbb R)$ 
and $\mathcal A$ is an elliptic operator. % of the form
%\begin{equation}
%\label{eq:calAop}
% \mathcal A u :=  -\partial_i[a_{ij}(x) \partial_j u + a_i(x) u]
%                     + b_i(x) \partial_i u  + c_0(x) u.
%\end{equation} 
%In the above, we use summation convention with $i,j$
%running from $1$ to $N$. Also, we assume that $a_{ij}, a_i, b_i$ and $c_0$ are
%functions in $L^{\infty}(\Omega)$ and that there exists a constant $\alpha_0 > 0$
%independent of $x \in \Omega$ such that
%\begin{equation}
%\label{eq:ellipticity}
%   a_{ij}(x) \xi_i \xi_j \geq \alpha_0 |\xi|^2,
%\end{equation} 
%for all $\xi \in \mathbb R^N$.
Our aim is to study how dynamics of the parabolic equation \eqref{eq:paraInvMan} changes when
we vary the domain $\Omega$. In particular, we wish to establish the continuity 
of invariant manifolds with respect to the domain.   
We will consider a sequence of uniformly bounded domains $\Omega_n$ in $\mathbb R^N$ as a perturbation
of $\Omega$. The perturbation of \eqref{eq:paraInvMan} is given by
\begin{equation}
\label{eq:paraInvManN}
  \left\{ 
  \begin{aligned}
    \frac{\partial u}{\partial t} + \mathcal A_n u  &=g_n(x,u) &&\quad \text { in } \Omega_n \times (0,\infty)\\
    u &= 0  &&\quad \text{ on } \partial \Omega_n \times (0,\infty)\\
     u(\cdot,0) &= u_{0,n} &&\quad \text{ in } \Omega_n. \\
  \end{aligned} 
  \right . 
\end{equation}
%where $\mathcal A_n$ is defined similarly to \eqref{eq:calAop}.
%
We impose conditions on the nonlinearities $g_n$ and $g$ so that the corresponding abstract parabolic 
equations 
\begin{equation}
 \label{eq:absSemiLinInvarN}
 \begin{aligned}
 \left \{
 \begin{aligned}
   \dot{u}(t) + A_n u(t) &= f_n(u(t)) \quad t \in (0, \infty)\\
         u(0) &= u_{0,n}, \\ 
 \end{aligned}
 \right .
 \end{aligned}
\end{equation}
where $f_n(u)(x):= g_n(x,u(x))$ and
\begin{equation}
 \label{eq:absSemiLinInvar}
 \begin{aligned}
 \left \{
 \begin{aligned}
   \dot{u}(t) + A u(t) &= f(u(t)) \quad t \in (0, \infty)\\
         u(0) &= u_{0}, \\ 
 \end{aligned}
 \right .
 \end{aligned}
\end{equation}
where $f(u)(x):= g(x,u(x))$  
are well-posed in $L^2(\Omega_n)$ and $L^2(\Omega)$, respectively.
In addition, we assume that $f_n(u)$ and $f(u)$ are higher order terms, that is, we will
consider \eqref{eq:absSemiLinInvarN} and \eqref{eq:absSemiLinInvar} as the linearised systems 
near an equilibrium (see Assumption \ref{assump:nonlinearFInvar}).

In this work, we focus on \emph{singular} perturbations of the domain, e.g. its topology
changes, so that it is not possible in general to apply a change of variables (coordinate
transform) to change the perturbed equation into an equivalent problem over the same
spatial domain $\Omega$.  This means that our class of domain perturbations cannot be reduced 
to a classical perturbation for the coefficients. Common examples include a sequence of 
dumbbell shape domains with shrinking handle and a sequence of domains with cracks.
One of the main difficulties to establish the persistence result under domain perturbation
is that the solutions of parabolic equations belong to different spaces, namely, 
$L^2(\Omega_n)$ and consequently the dynamical systems (semiflows) induced by these parabolic 
equations act on different spaces.  

It is well-known from the theory of dynamical systems that hyperbolicity of an equilibrium
is the main concept for persistence under small perturbations. We show in this paper that
this principle is also valid for singular domain perturbation.
Our main result states that under a suitable rather general class of domain perturbation, 
if the unperturbed system \eqref{eq:absSemiLinInvar}
has a local stable %invariant manifold $W^s$ 
and a local unstable invariant manifolds %$W^u$
in a neighbourhood of an equilibrium and the equilibrium is hyperbolic, then
the perturbed system \eqref{eq:absSemiLinInvarN} also has a local stable %invariant manifold $W_n^s$ 
and a local unstable invariant manifolds for $n$ sufficiently large. %$W_n^u$. 
Moreover, we have the continuity (upper and lower semicontinuity) of these invariant manifolds with respect to the domain
(see Theorem \ref{th:mainResults} and Theorem \ref{th:mainResults2}). This result is new.

% There are similar results on the effect of domain variation on the dynamics of parabolic equations. 
% In \cite{MR2130211}, upper and lower semicontinuity of attractors are obtained 
% for the heat equation subject to Dirichlet boundary condition under a certain perturbation of the domain
% in $\mathbb R^N$ with $N \leq 4$.
% The result is proved via the convergence of spectrum and semigroup of the corresponding elliptic operators.
% Similarly, in \cite{MR1834117}, upper semicontinuity of attractors is obtained for
% reaction-diffusion equations with Neumann boundary condition when the domain $\Omega \subset \mathbb R^M
% \times \mathbb R^N$ is squeezed in the $\mathbb R^N$-direction. 
% In \cite{MR2041515}, a necessary and sufficient condition on
% domains for spectral convergence of elliptic equations subject to Neumann boundary condition is obtained.
% Consequently, the continuity of local unstable invariant manifolds and the continuity of attractors are then proved
% under this condition on domains.
There are similar results on the effect of domain variation on the dynamics of parabolic equations. 
In \cite{MR1834117}, upper semicontinuity of attractors is obtained for
reaction-diffusion equations with Neumann boundary condition when the domain $\Omega \subset \mathbb R^M
\times \mathbb R^N$ is squeezed in the $\mathbb R^N$-direction. Arrieta and Carvalho \cite{MR2041515}   
consider a similar problem on a sequence of bounded and Lipschitz perturbed
domains $\Omega_n$. They give necessary and sufficient conditions on
domains for spectral convergence of the corresponding elliptic problem and obtain continuity (upper and lower
semicontinuity) of local unstable manifolds and consequently continuity of attractors. 
For results under Dirichlet boundary condition, we refer to \cite{MR2130211}
where upper and lower semicontinuity of attractors are obtained 
for the heat equation under a certain perturbation of the domain in $\mathbb R^N$ with $N \leq 4$.

The class of domain perturbations considered in this paper (Assumption 
\ref{assum:bddPertMosco}) is much more general than that in \cite{MR2130211}.
Many examples where this more general domain convergence is useful appear in \cite{MR1955096}
(as well as many other references).
These have been used in constructing many examples of domains where the time independent 
problem is much more complicated than when $\Omega$ is a ball. 
For this general class of domain perturbations, we also have prior knowledge of the 
convergence of eigenvalues and eigenfunctions of the
corresponding elliptic operators. The main focus here is to investigate the dependence of
domains in the construction of invariant manifolds. 
In \cite{MR2041515}, continuity of local unstable invariant manifolds is proved by keeping
track of the construction adapted from Henry \cite{MR610244}.
Although our framework on semilinear parabolic equations fits into \cite{MR610244}, we will 
use different techniques. Indeed, we apply the existence results for invariant manifolds in Bates and Jones \cite{MR1000974} to 
prove the continuity of invariant manifolds under domain perturbation.     
The construction of invariant manifolds in \cite{MR1000974} follows Hadamard style \cite{MR1508998}
which involves using the splitting between various subspaces to estimate projections of the 
flow in the different directions. 
The technique involves more geometrical than functional-analytic arguments.
By using this construction, we give
continuity results for both the local stable and the local unstable invariant manifolds
under domain perturbation rather than focus only on the local unstable invariant manifolds 
(and consequently attractors) as in \cite{MR2130211, MR1834117, MR2041515}.

An outline of this paper is as follows. 
% Section \ref{sec:prelimInvMan} provides a basic
% background of invariant manifolds for semiflows induced by solutions of
% semilinear parabolic equations. In particular, we review the
% construction of invariant manifolds given by Bates and Jones
% \cite{MR1000974}. 
In Section \ref{sec:frameworkMain}, we state our framework and the main results on the 
continuity (upper and lower semicontinuity) of the local stable and the local unstable
invariant manifolds under perturbation of the domain. 
In Section \ref{sec:existInvPert}, we obtain the existence of local invariant manifolds for 
the perturbed problems following the construction from \cite{MR1000974}.
In Section \ref{sec:technicalLemma},
we give some technical lemmas and a characterisation of upper and lower
semicontinuity. The proof of the continuity results is given in Section \ref{sec:convUnstaMan}
for the local unstable invariant manifolds and in Section \ref{sec:convStaMan} for the local 
stable invariant manifolds.

\section{Framework and main results}
\label{sec:frameworkMain}

%\subsection{Assumptions and frameworks for domain perturbation}
%\label{subsec:semiFlowsByAbsEq}
Let $\Omega_n$ be a sequence of bounded open sets in $\mathbb R^N$, $N \geq 2$
and $\Omega$ be a bounded open set in $\mathbb R^N$ such that 
there exists a ball $D \subset \mathbb R^N$ with $\Omega_n, \Omega \subset D$ 
for all $n \in \mathbb N$. 
We consider the
perturbed semilinear parabolic equation \eqref{eq:paraInvManN} 
%initial boundary value problem
% \begin{equation}
% \label{eq:paraInvMan}
%   \left\{ 
%   \begin{aligned}
%     \frac{\partial u}{\partial t} + \mathcal A u  &=g(x,u) &&\quad \text { in } \Omega \times (0,\infty)\\
%     u &= 0  &&\quad \text{ on } \partial \Omega \times (0,\infty)\\
%      u(\cdot,0) &= u_0 &&\quad \text{ in } \Omega , \\
%   \end{aligned} 
%   \right . 
% \end{equation}
% where $g \in C(\mathbb R^N \times \mathbb R)$, and 
where $\mathcal A_n$ is an elliptic operator of the form
\begin{equation}
\label{eq:calAopN}
 \mathcal A_n u :=  -\partial_i[a_{ij,n}(x) \partial_j u + a_{i,n}(x) u]
                     + b_{i,n}(x) \partial_i u  + c_{0,n}(x) u .
\end{equation} 
In the above, we use summation convention with $i,j$
running from $1$ to $N$. Also, we assume  $a_{ij,n}, a_{i,n}, b_{i,n}, c_{0,n}$ are
functions in 
$L^{\infty}(D)$ and that there exists a constant $\alpha_0 > 0$
independent of $x \in D$ and $n \in \mathbb N$ such that
\begin{equation}
\label{eq:ellipticity}
a_{ij,n}(x) \xi_i \xi_j \geq \alpha_0 |\xi|^2,
\end{equation} 
for all $\xi \in \mathbb R^N$ and for all $n \in \mathbb N$.
The elliptic operator $\mathcal A$ for the unperturbed equation \eqref{eq:paraInvMan} is  
defined similarly to \eqref{eq:calAopN} (with $n$ deleted) and $a_{ij}$ satisfies the
ellipticity condition \eqref{eq:ellipticity} with the same constant $\alpha_0$.
We assume that the coefficients of the operator $\mathcal A_n$
converge to the corresponding coefficients of $\mathcal A$ as stated below.
\begin{assumption}
\label{assum:An}
Assume that
    $\lim_{n\rightarrow \infty}
    a_{ij,n}=  a_{ij}$, $\lim_{n\rightarrow \infty} a_{i,n} = a_{i},
    \lim_{n\rightarrow \infty} b_{i,n} = b_{i}$ and
    $\lim_{n\rightarrow \infty}c_{0,n} = c_0$ in $L^{\infty}(D)$
    for all $i,j = 1, \ldots, N$. 
\end{assumption}

By Riesz representation theorem, we identify $L^2(\Omega_n)$ with its dual
and consider the \emph{evolution triple}  
$H^1_0(\Omega_n) \overset{d}{\hookrightarrow} L^2(\Omega_n) \overset{d}{\hookrightarrow} H^{-1}(\Omega_n)$.
We denote by $\langle \cdot, \cdot \rangle$  the duality pair between 
$H^{-1}(\Omega_n)$ and $H^1_0(\Omega_n)$.
The notation $( \cdot \; | \; \cdot )_{L^2(\Omega_n)}$ denotes
the inner product on $L^2(\Omega_n)$.  
Define a form $a_n(\cdot,\cdot)$ associated with $\mathcal A_n$ on $H^1_0(\Omega_n)$ by
\begin{equation}
 \label{eq:bilinearForm}
  a_n(u,v) :=  \int_{\Omega_n}[ a_{ij,n}(x) \partial_j u+ a_{i,n}(x)u] \partial_i v
                + b_{i,n}(x) \partial_i u v + c_{0.n}(x) uv dx,
\end{equation}
for $u,v \in H^1_0(\Omega_n)$. It is easy to see that $a_n(\cdot,\cdot)$ is
a continuous bilinear form. We define $a(\cdot,  \cdot)$ on $H^1_0(\Omega)$ similarly.
Let
\begin{equation}
 \label{eq:constLambda0}
  \lambda_{\mathcal A} := \sup_{n \in \mathbb N} \Big \{ \|c_{0,n}^{-}\|_{\infty} + \frac{1}{2 \alpha_0} \sum_{i=1}^N
  \|a_{i,n} + b_{i,n} \|_{\infty} \Big \},
\end{equation}
where $c_{0,n}^{-}:= \max(-c_{0.n}, 0)$ is the negative part of $c_{0,n}$. We set 
$\lambda_0 := \lambda_{\mathcal A} + \alpha_0 / 2$.
It can be verified that 
\begin{equation}
 \label{eq:ellipticCond}
  a_n(u,u) + \lambda \|u\|^2_{L^2(\Omega_n)} \geq \frac{\alpha_0}{2}\|u\|^2_{H^1_0(\Omega_n)},
\end{equation}
for all $u \in H^1_0(\Omega_n)$, for all $\lambda \geq \lambda_0$ and for all 
$n \in \mathbb N$. Similar inequality holds for $a(\cdot, \cdot)$ with the same constants.
By the Lax--Milgram theorem, there exists $A_{\Omega_n} \in \mathscr L
(H^1_0(\Omega_n), H^{-1}(\Omega_n))$ such that 
\begin{equation}
 \label{eq:operatorAonV}
  a_n(u,v) = \langle A_{\Omega_n} u, v \rangle,
\end{equation}
for all $u,v \in H^1_0(\Omega_n)$. 
We may consider $A_{\Omega_n}$ as an operator on $H^{-1}(\Omega_n)$ with the
domain $H^1_0(\Omega_n)$. Similarly, we obtain the operator
$A_{\Omega} \in \mathscr L(H^1_0(\Omega), H^{-1}(\Omega))$. 
% Denote the resolvent sets
% of $-A_{\Omega_n}$ and $-A_{\Omega}$ by $\varrho(-A_{\Omega_n})$ and $\varrho(-A_{\Omega})$,
% respectively.
% A standard argument shows that $[\lambda_0, \infty) \subset \varrho(-A_{\Omega_n}) 
% \cap\varrho(-A_{\Omega})$ for all $n \in \mathbb N$.
% (see \cite[Section 5.1]{Daners20081}). 
Let $A_n$ and $A$ be the \emph{maximal restriction} of the operators $A_{\Omega_n}$
and $A_{\Omega}$ on $L^2(\Omega_n)$ and $L^2(\Omega)$, respectively. 
%We have that 
% \begin{displaymath} 
% [\lambda_0, \infty) \subset
% \varrho(-A_{\Omega_n}) \cap \varrho(-A_{\Omega}) \subset \varrho(-A_n) \cap \varrho(-A)
%\end{displaymath}
% for all $n \in \mathbb N$.
It is well-known that $-A_n$ generates a strongly continuous analytic semigroup $S_n(t), t
\geq 0$ on $L^2(\Omega_n)$
(see \cite[Proposition 3, XVII \S 6]{MR1156075}). Similarly, we denote by $S(t), t
\geq 0$ the semigroup on $L^2(\Omega_n)$ generated by $-A$. 
% Moreover, there exist $C > 0$ and  $\omega \in \mathbb R$ 
% such that
% \begin{displaymath}
%   \|S(t)\|_{\mathscr L(L^2(\Omega))} \leq C e^{\omega t}
% \end{displaymath}
% for $t \geq 0$.
% We can write \eqref{eq:paraInvMan} in the abstract form as 
% \begin{equation}
%  \label{eq:absSemiLinInvar}
%  \begin{aligned}
%  \left \{
%  \begin{aligned}
%    \dot{u}(t) + A u(t) &= f(u(t)) \quad t \in (0, \infty)\\
%          u(0) &= u_{0} \\ 
%  \end{aligned}
%  \right .
%  \end{aligned}
% \end{equation}   
% in $L^2(\Omega)$, where $f(u)(x):= g(x,u(x))$ is the substitution
% operator induced by $g$.]
We shall consider the perturbation \eqref{eq:paraInvManN} of \eqref{eq:paraInvMan} 
in the abstract form \eqref{eq:absSemiLinInvarN} and 
\eqref{eq:absSemiLinInvar} in $L^2(\Omega_n)$ and $L^2(\Omega)$, respectively. 

% When well posed, solution of \eqref{eq:absSemiLinInvar} can
% be represented by the variation of constant formula
% \begin{equation}
% \label{eq:varOfConst}
%   u(t) = S(t)u_0 + \int_0^t S(t - \tau) f(u(\tau)) d\tau.
% \end{equation}
% A solution corresponding to the above integral equation
% \eqref{eq:varOfConst} is often called a \emph{mild solution}.

To deal with domain perturbation where the solutions belong to different function spaces, 
we usually consider the trivial extension, that is, the 
extension by zero on $D \backslash \Omega$. 
In abuse of notation, we often write 
$u \in L^2(D)$ for the trivial extension of a function $u \in L^2(\Omega)$.
On the other hand, we write $u \in L^2(\Omega)$ for a function $u \in L^2(D)$ to represent its 
restriction to $\Omega$. 
In particular, when we write $u_n \rightarrow u$ in $L^2(D)$ for $u_n \in L^2(\Omega_n)$
we mean the trivial extensions converge in $L^2(D)$. 
The notation $u_n |_{\Omega}$ where $u_n \in L^2(\Omega_n)$ means that
$u_n$ is first extended by zero on $D \backslash \Omega_n$ and then restricted to $\Omega$.
A similar interpretation applies to the notation $u|_{\Omega_n}$ when $u \in L^2(\Omega)$. 
We will use this convention throughout the paper without further comment.

We assume that a sequence of domains $\Omega_n$ converges to $\Omega$
in the following sense. 
\begin{assumption}
\label{assum:bddPertMosco}
We assume the following two conditions hold:
\begin{itemize}
 \item[\upshape{(M1)}]
  For every $\phi \in H^1_0(\Omega)$, there exists $\phi_n$ in $H^1_0(\Omega_n)$
  such that $\phi_n \rightarrow \phi$ in $H^1(D)$. 
 \item[\upshape{(M2)}]
  If $(n_k)$ is a sequence of indices converging to $\infty$, $(\phi_{n_k})$
  is a sequence with $\phi_{n_k} \in H^1_0(\Omega_{n_k})$ and $\phi_{n_k} \rightharpoonup \phi$
  in $H^1(D)$ weakly, then the weak limit $u$ belongs to $H^1_0(\Omega)$.
\end{itemize}
\end{assumption} 
Note that here we regard $H^1_0(\Omega_n)$ and $H^1_0(\Omega)$ as closed subspaces of $H^1(D)$
using the trivial extension. It is often said that
$H^1_0(\Omega_n)$ converges to $H^1_0(\Omega)$ in the sense of Mosco when (M1) and (M2) hold, 
but we will simply say that $\Omega_n$ converges to $\Omega$ in sense of Mosco. 
We refer to \cite{MR0298508} for a general Mosco convergence of closed convex sets. 
Examples of domains satisfying (M1) and (M2) can be found in \cite{MR1955096}.   

The Mosco convergence assumption  is 
naturally used in domain perturbation. As characterised in \cite{MR1955096}, it is a
necessary and sufficient condition for strong convergence and uniform convergence of the 
resolvent operators under domain perturbation. It is also sufficient for the convergence of 
solutions of initial value problems for parabolic equations (see \cite[Section 6]{MR2119988}).

We make the following assumption on the nonlinearities.
\begin{assumption}
\label{assump:nonlinearFInvar} 
We assume that  
\begin{itemize}
\item[\upshape{(i)}]
$f: L^2(\Omega) \rightarrow
L^2(\Omega)$ is locally Lipschitz and $f(0) = 0$.  Moreover, for every
$\varepsilon > 0$ there exists a neighbourhood $U =U(\varepsilon)$ of $0$ such that
$f$ has a Lipschitz constant $\varepsilon$ in $U$.

\item[\upshape{(ii)}]  $f_n: L^2(\Omega_n) \rightarrow  L^2(\Omega_n)$ 
    is locally Lipschitz and $f_n(0) = 0$. In addition, for every
   $\varepsilon > 0$ there exists a neighbourhood $U_n =U_n(\varepsilon)$ of $0$ such that
   $f_n$ has a Lipschitz constant $\varepsilon$ in $U_n$. Moreover,  
   $U_n$ can be chosen uniformly with respect to $n \in \mathbb N$ in
   the sense that we can take $U_n$ to be a ball centered at $0$ in $L^2(\Omega_n)$
   of the same radius for all $n \in \mathbb N$.
   
\item[\upshape{(iii)}]  $f_n(u|_{\Omega_n}) \rightarrow f (u|_{\Omega})$ in $L^2(D)$ uniformly
   with respect to $u \in B_{L^2(D)}(0,r)$ for all $r > 0$.
\end{itemize}
\end{assumption}

\begin{remark}
\label{rem:conditionOnSubOp}
(i) Assumption \ref{assump:nonlinearFInvar} (i) means that $f(u)$ is a higher order term
and we could think of \eqref{eq:absSemiLinInvar} as a linearised problem near an equilibrium.

(ii)  A necessary and sufficient condition for the substitution operator $f$ to be in
$C(L^p(\mathbb R^N),L^q(\mathbb R^N))$ is that
 there exist  $c>0$ and $\psi \in L^q(\mathbb R^N)$ such that
 %\begin{displaymath}
    $|g(x,\xi)| \leq \psi(x) + c|\xi|^{p/q}$
 %\end{displaymath}
for all $x \in \mathbb R^N$ and $\xi \in \mathbb R$ (see \cite{MR1066204}). 
Hence, Assumption \ref{assump:nonlinearFInvar} (i.e. $p = q = 2$) means that we
require a linear growth with respect to $u$ in the nonlinear term $g(x,u)$.

(iii) The Lipschitz continuity of $f$ is for instance satisfied if 
there exists an essentially bounded function $\phi$ such that
  %\begin{displaymath}
    $|g(x,\xi_1) - g(x, \xi_2)| \leq \phi(x, R) |\xi_1-\xi_2|$
  %\end{displaymath}  
for all $|\xi_1|, |\xi_2| \leq R$ (see \cite[Theorem 3.10]{MR1066204}).

(iv) The condition $f(0) = 0$ holds if $g(x, 0) = 0$ for almost all $x \in \Omega$.
\end{remark}

By our assumptions on $A_n$ and $f_n$, the abstract equation \eqref{eq:absSemiLinInvarN} has 
a unique mild solution $u_n \in C([0, t_n^{+}(u_{0,n})), L^2(\Omega_n))$ for any given initial 
condition $u_{0,n} \in L^2(\Omega_n)$ (see \cite{MR710486} or \cite[Theorem 3.8]{MR2119988}). 
Here, we write $t_n^{+}(u_{0,n})$ for the \emph{maximal existence time} or 
\emph{positive escape time}. Moreover, the mild solution $u_n$ of \eqref{eq:absSemiLinInvarN}
can be represented by the \emph{variation of constants} formula
\begin{equation}
\label{eq:varOfConst}
  u_n(t) = S_n(t)u_{0,n} + \int_0^t S_n(t - \tau) f_n(u_n(\tau)) d\tau,
\end{equation}
for $t \in [0,  t_n^{+}(u_{0,n}))$.
Since $g_n$ is linearly bounded with respect to the second variable 
(Remark \ref{rem:conditionOnSubOp} (ii)), we have that $t_n^{+}(u_{0,n}) = \infty$ for 
all $u_{0,n} \in L^2(\Omega_n)$, that is, we always have a \emph{global} solution.
Similar consideration implies the existence and uniqueness of mild solution $u$ of  
\eqref{eq:absSemiLinInvar}.   

To study the abstract parabolic equation as a dynamical system, we consider
a \emph{semiflow} $\Phi_{t,n} : L^2(\Omega_n) \rightarrow L^2(\Omega_n)$
defined by
\begin{equation}
 \label{eq:semiflow}
 \Phi_{t,n} (u_{0,n}) := u_n(t),
\end{equation}
for all $t \in [0,t_n^{+}(u_{0,n}))$ where $u_n$ is the maximal solution of 
\eqref{eq:absSemiLinInvarN}. 
Sometimes we would like to study the backwards behaviour of solutions. We call a
continuous curve $u_n: [-t,0] \rightarrow L^2(\Omega_n)$ for some $t > 0$ a \emph{backwards
solution branch} for $u_{0,n} \in L^2(\Omega_n)$  if
%\begin{displaymath}
  $\Phi_{s,n}(u_n(-s)) = u_{0,n}$
%\end{displaymath}
for all $s \in [0,t]$. We write $\Phi_{-s,n} (u_{0,n}) = u_n(-s)$ when we look at
a particular backwards solution branch. We defined the semiflow 
$\Phi_t : L^2(\Omega) \rightarrow L^2(\Omega)$ induced by
solutions of \eqref{eq:absSemiLinInvar} similarly. 

Under the assumptions considered above, it is proved in 
\cite[Theorems 1.1 (i), 1.2 (i) ]{MR1000974} that the unperturbed problem
\eqref{eq:absSemiLinInvar} has a local stable invariant manifold $W^s$ and
a local unstable invariant manifold $W^u$ inside a suitable neighbourhood $U$ of $0$
(see Section \ref{subsec:constructInv}).
In this paper, we study the persistence of these local invariant manifolds
 under domain perturbation when the equilibrium $0 \in L^2(\Omega)$ of 
\eqref{eq:absSemiLinInvar} is \emph{hyperbolic}, that is, the spectrum $\sigma(-A)$ of $-A$
does not contain $\lambda$ with $\real{\lambda} = 0$. 

The main results of this paper can be stated as follows.
\begin{theorem}[Continuity of local unstable manifolds]
\label{th:mainResults}
Suppose that Assumption \ref{assum:An}, \ref{assum:bddPertMosco} and \ref{assump:nonlinearFInvar}
are satisfied. If the equilibrium $0$ of \eqref{eq:absSemiLinInvar} is hyperbolic,
then \eqref{eq:absSemiLinInvarN} has a local unstable invariant manifold $W_n^{u}$ for
$n$ sufficiently large such that there exists $\delta > 0$ for which the following
(i) and (ii) hold.  
% parabolic equation
% \eqref{eq:absSemiLinInvarN} for each $n \in \mathbb N$. Moreover, 
% there exists  $\delta > 0$ such that if we denote by 
%  $B_n = B_{L^2(\Omega_n)}(0,\delta)$ and  
%  $B = B_{L^2(\Omega)}(0,\delta)$, then we have
 \begin{itemize}
  \item[\upshape{(i)}] Upper semicontinuity:
   \begin{displaymath}
      \sup_{v \in W_n^{u} \cap B_n} \inf_{ u \in W^{u} \cap B}
       \| v - u \|_{L^2(D)} \rightarrow 0
     %%% \sup_{v \in \graph(h_n^{*}) \cap U_n} \inf_{ u \in \graph(h^{*}) \cap U} \| i_n(v) -i(u) \|_{L^2(D)} \rightarrow  
      \quad \text{ as } n \rightarrow \infty;
   \end{displaymath}
     
  \item[\upshape{(ii)}] Lower semicontinuity:   
   \begin{displaymath}
     \sup_{u \in W^{u} \cap B} \inf_{ v \in W_n^{u} \cap B_n}
     \| v - u \|_{L^2(D)} \rightarrow 0 
      %%%\sup_{u \in \graph(h^{*}) \cap U} \inf_{ v \in \graph(h_n^{*}) \cap  U_n} \|i_n(v) -i(u) \|_{L^2(D)} \rightarrow  
      \quad \text{ as } n \rightarrow \infty,
   \end{displaymath}
   %as $n \rightarrow \infty$.  
 \end{itemize}
  where $B_n := B_{L^2(\Omega_n)}(0,\delta)$ and   $B := B_{L^2(\Omega)}(0,\delta)$.
\end{theorem}

A similar result can be stated for local stable invariant manifolds with an additional 
assumption of the convergence in measure of the domains. We denote by $|\Omega|$
the Lebesgue measure of $\Omega$.
\begin{theorem}[Continuity of local stable manifolds]
\label{th:mainResults2}
Suppose that Assumption \ref{assum:An}, \ref{assum:bddPertMosco} and \ref{assump:nonlinearFInvar}
are satisfied.  
In addition, assume that $|\Omega_n| \rightarrow |\Omega|$ 
as $n \rightarrow \infty$.
If the equilibrium $0$ of \eqref{eq:absSemiLinInvar} is hyperbolic,
then \eqref{eq:absSemiLinInvarN} has a local stable invariant manifold $W_n^{s}$ for
$n$ sufficiently large such that there exists $\delta > 0$ for which the following
(i) and (ii) hold.  
 \begin{itemize}
  \item[\upshape{(i)}] Upper semicontinuity:
   \begin{displaymath}
      \sup_{v \in W_n^{s} \cap B_n} \inf_{ u \in W^{s} \cap B}
       \| v - u \|_{L^2(D)} \rightarrow 0
       \quad \text{ as } n \rightarrow \infty;
   \end{displaymath}
  % as $n \rightarrow \infty$. 
  \item[\upshape{(ii)}] Lower semicontinuity:   
   \begin{displaymath}
     \sup_{u \in W^{s} \cap B} \inf_{ v \in W_n^{s} \cap B_n}
      \| v - u \|_{L^2(D)} \rightarrow 0  
      \quad \text{ as } n \rightarrow \infty,
   \end{displaymath}
  % as $n \rightarrow \infty$.  
 \end{itemize}
  where $B_n := B_{L^2(\Omega_n)}(0,\delta)$ and   $B := B_{L^2(\Omega)}(0,\delta)$.
\end{theorem}
% In Section \ref{subsec:existInvPert}, we show the existence of local unstable invariant manifolds
% and the existence of local stable invariant manifolds for the perturbed problem \eqref{eq:absSemiLinInvarN}
% stated in Theorem \ref{th:mainResults} and  Theorem \ref{th:mainResults2}, respectively.
% The assertion on upper and lower semicontinuity will be proved in Section \ref{sec:convUnstaMan}
% for unstable invariant manifolds and in Section \ref{sec:convStaMan} for stable invariant manifolds.

\section{Existence of invariant manifolds for the perturbed equations}
\label{sec:existInvPert}

In this section, we obtain the existence of local unstable and local stable invariant
manifolds for the perturbed equation \eqref{eq:absSemiLinInvarN} stated in Theorem 
\ref{th:mainResults} and  Theorem \ref{th:mainResults2} using the construction 
from \cite{MR1000974}. For the sake of mathematical necessity, we first give a sketch
of proof of the existence of invariant manifolds proved in \cite{MR1000974}
for the unperturbed equation \eqref{eq:absSemiLinInvar}. We then keep track of
this construction to obtain invariant manifolds for the perturbed equations. 
   
\subsection{The construction of invariant manifolds}
\label{subsec:constructInv}

\begin{definition}
\label{def:staUnstaInvManifold}
Let $U$ be a neighbourhood of $0$. We define
\begin{displaymath}
 \begin{aligned}
  W^{s}&= \{ u \in U : \Phi_t(u) \in U \text{ for all } t
        \geq 0 \text { and } \Phi_t(u) \rightarrow 0 \text{ exponentially as }
         t \rightarrow \infty \} \\
 W^{u} &= \{u \in U : \text{ some backwards branch } \Phi_t(u) \text{
   exists for all } t < 0 \text{ and lies in } U, \\
   &\quad \quad \text{ and }  \Phi_t(u) \rightarrow 0 \text{ exponentially as }  t \rightarrow -\infty \}
 \end{aligned}
\end{displaymath}
\end{definition}
These sets $W^{s}$ and $W^{u}$ are invariant relative to $U$ and are
called stable and unstable sets, respectively. Under the assumptions in Section 
\ref{sec:frameworkMain}, 
it is proved in \cite{MR1000974} that $W^{s}$ and $W^{u}$ are indeed invariant manifolds
for the unperturbed problem \eqref{eq:absSemiLinInvar}.
We sometimes write $W^s(U)$ and $W^u(U)$ to indicate their dependence on the 
neighbourhood $U$. 

Recall from Section \ref{sec:frameworkMain} that $-A$ is a generator of an analytic 
$C_0$-semigroup $S(t), t \geq 0$ on $L^2(\Omega)$. 
%We denote the spectrum of $-A$ by $\sigma(-A)$ and decompose it as
We decompose the spectrum $\sigma(-A)$ as
\begin{displaymath}
  \sigma(-A) = \sigma^{s} \cup \sigma^c \cup \sigma^u
\end{displaymath}  
where
\begin{equation}
\label{eq:sigmaSCU}
 \begin{aligned}
  \sigma^s &= \{ \lambda \in \sigma(-A) : \mathrm{Re}(\lambda) < 0 \}\\
  \sigma^c &= \{ \lambda \in \sigma(-A) : \mathrm{Re}(\lambda) = 0 \}\\
  \sigma^u &= \{ \lambda \in \sigma(-A) : \mathrm{Re}(\lambda) > 0 \}.
 \end{aligned}
\end{equation}
Since $\Omega$ is bounded, Rellich's theorem implies that the embedding
$H^1_0(\Omega) \hookrightarrow L^2(\Omega)$ is compact. Hence, the
resolvent $(\lambda + A)^{-1} : L^2(\Omega) \rightarrow L^2(\Omega)$ is
also compact when it is defined. This implies that $\sigma(-A)$ consists
of eigenvalues with finite multiplicities (see \cite{MR0407617}). 
It is easily seen from \cite[Theorem 3, XVII \S 6]{MR1156075} 
that $\sigma^c$ and $\sigma^u$ are finite sets. 
Let $\Gamma^c$ and $\Gamma^u$ be rectifiable closed curves 
separating $\sigma^c$ and $\sigma^s$ respectively from the remaining
spectrum. There are invariant subspaces of $L^2(\Omega)$
associated to $\sigma^s, \sigma^c$ and $\sigma^u$ via the spectral
projections (see \cite{MR0407617})
\begin{equation}
\label{eq:spectralProjs}
 %\begin{aligned} 
 P^c  = \frac{1}{2\pi i} \int_{\Gamma^c} (\lambda + A)^{-1} d\lambda
 \quad \text{ and } \quad 
 P^u  = \frac{1}{2\pi i} \int_{\Gamma^u} (\lambda + A)^{-1} d\lambda.
 %\end{aligned}
\end{equation}
Indeed, we decompose $L^2(\Omega) = X^s \oplus X^c \oplus X^u$ where
$X^s = (1-P^c -P^u) L^2(\Omega)$, $X^c = P^c L^2(\Omega)$ and $X^u = P^u
L^2(\Omega)$. Note that $\dim(X^c)$ and $\dim(X^u)$ are finite.
We  set $X^{cs} =X^c \oplus X^s$ and $X^{cu} = X^c \oplus
X^u$. For $* = s,c,u,cs,cu$, we have that $-A^{*} = -A|_{X^{*}}$ is a
generator of $S^{*}(t) = S(t)|_{X^{*}}$.  Since $S(t)$ is an analytic semigroup,
there exist $M > 0$ and $\sigma > 0$ such that
%\begin{displaymath}
 $\|S^{s}(t)\| \leq M e^{-\sigma t}$
%\end{displaymath}
for all $t > 0$.

To obtain the existence of local stable and unstable invariant manifolds, we
decompose $L^2(\Omega) = X^{-} \oplus X^{+}$ with 
$\dim{X^{+} }< \infty$ in two different ways; 
either $X^{-} =X^{s}$ and  $X^{+} = X^{cu}$, or 
$X^{-} =X^{cs}$ and  $X^{+} = X^{u}$. 
We denote a natural projection (via spectral projection) onto $X^{+}$ by $P^{+}$,
a natural projection on $X^{-}$ by $P^{-} := 1-P^{+}$ and write $-A^{\pm} = -A|_{X^{\pm}}$.  
In both cases, we have that    
 $-A^{-}$ generates a $C_0$-semigroup $S^{-}(t)$ on $X^{-}$ satisfying
         \begin{equation}
          \label{eq:alphaSminus}
          \|S^{-}(t)\| \leq M_1 e^{\alpha t},
         \end{equation}
         for all $t \geq 0$ where $M_1 > 0$ and $\alpha \in \mathbb R$.
 Similarly, $-A^{+}$ generates a $C_0$-group $S^{+}(t)$ on $X^{+}$ satisfying
         \begin{equation}
          \label{eq:betaSplus}
          \|S^{+}(t)\| \leq M_2 e^{\beta t},
         \end{equation}
         for all $t \leq 0$ where $M_2 > 0$ and $\beta > \alpha$.
The parameters $\alpha$ and $\beta$ can be chosen as follows 
(see proof of Theorem 1.1 case (D) and proof of Theorem 1.2 case (D) in \cite{MR1000974}).
\begin{itemize}
 \item If $X^{-} =X^{s}$ and  $X^{+} = X^{cu}$, we take $\alpha = -\sigma$ and fix 
       $\beta$ such that $-\sigma < \beta < 0$.
 \item If $X^{-} =X^{cs}$ and  $X^{+} = X^{u}$, we take $\beta > 0$ such that
       $\beta < \min\{\mathrm{Re}(\lambda) : \lambda \in \sigma^u \}$ and fix
       $\alpha$ such that $0 < \alpha < \beta$.
\end{itemize}
The main techniques used in \cite{MR1000974} are a renorming of $X^{-}$ and
$X^{+}$ and a modification of nonlinearity $f$. Since we decompose 
$L^2(\Omega) = X^{-} \oplus X^{+}$, norms on $X^{-}$ and $X^{+}$ are
originally inherited from $L^2(\Omega)$. Indeed, if $u = v \oplus w \in
L^2(\Omega)$ where $v \in X^{-}$ and  $w \in X^{+}$, then
\begin{equation}
 \label{eq:equivNormL2}
   \frac{1}{\|P^{-} \| +\|P^{+}\|} (\|v\|_{L^2(\Omega)} + \|w\|_{L^2(\Omega)})
   \leq \|u\|_{L^2(\Omega)} 
   \leq (\|v\|_{L^2(\Omega)} + \|w\|_{L^2(\Omega)}).
\end{equation}   
However, we can renorm $X^{-}$ and $X^{+}$ by
\begin{equation}
\label{eq:normXminusplus}
\begin{aligned}
  \|v\|_{X^{-}} &:= \sup_{t \geq 0} e^{-\alpha t} \|S^{-}(t) v\|_{L^2(\Omega)}
                \quad && \text{ for } v \in X^{-}, \\
  \|w\|_{X^{+}} &:= \sup_{t \leq 0} e^{-\beta t} \|S^{+}(t) w\|_{L^2(\Omega)}
                \quad && \text{ for } w \in X^{+}.
\end{aligned}  
\end{equation}
%for $v \in X^{-}$ and similarly
% \begin{equation}
% \label{eq:normXplus} 
%   \|w\|_{X^{+}} := \sup_{t \leq 0} e^{-\beta t} \|S^{+}(t) w\|_{L^2(\Omega)},
% \end{equation}
% for $w \in X^{+}$.
These norms are equivalent on $X^{-}$ and $X^{+}$, respectively. It is
easy to see that (see also \cite[Lemma 2.1]{MR1000974} )
\begin{equation}
\begin{aligned} 
 \label{eq:equivNormXminusplus}
   &\|v\|_{L^2(\Omega)} \leq \|v\|_{X^{-}} \leq M_1 \|v\|_{L^2(\Omega)}
               &&\quad \text{ for all } v \in X^{-}, \\ 
   &\|w\|_{L^2(\Omega)} \leq \|w\|_{X^{+}} \leq M_2 \|w\|_{L^2(\Omega)}
               &&\quad \text{ for all } w \in X^{+}.
\end{aligned}
\end{equation}
% \begin{equation}
%  \label{eq:equivNormXplus}
%    \|w\|_{L^2(\Omega)} \leq \|w\|_{X^{+}} \leq M_2 \|w\|_{L^2(\Omega)},
% \end{equation}
% for all $w \in X^{+}$.
The modification of nonlinearity $f$  is done by cutting off arguments so
that we obtain a globally Lipschitz function $\tilde f$.
Let $\eta > 0$ be arbitrary. By Assumption
\ref{assump:nonlinearFInvar}, we can choose $\delta > 0$ such that $f$ has a
Lipschitz constant less than $\eta/12$ in $B_{L^2(\Omega)}(0, 2\delta)$.
Let $\Psi: L^2(\Omega) \rightarrow \mathbb R$ be a function defined by
\begin{displaymath}
  \Psi(u) = 
            \left \{ 
              \begin{aligned} 
               &1  &&\quad \text{ if } \|u\|_{L^2(\Omega)} \leq \delta\\
               &2 - \frac{\|u\|_{L^2(\Omega)}}{\delta}  &&\quad \text{ if }
                  \delta \leq \|u\|_{L^2(\Omega)} \leq 2 \delta\\
               &0  &&\quad \text{ if } \|u\|_{L^2(\Omega)} \geq 2 \delta.
               \end{aligned}
            \right . 
\end{displaymath}
By setting $\tilde f(u) := \Psi(u) f(u)$ for all $u \in L^2(\Omega)$, we have
that $\tilde f$ is globally Lipschitz continuous with constant $\varepsilon <
\eta/4$. This Lipschitz constant $\varepsilon$ can be chosen as small as
we require by shrinking $\delta$.

With this modified system $\dot{u}(t) + A u(t) = \tilde f(u(t))$, the
solution to an initial value parabolic equation $u(t)$ also exists for $t
\geq 0$, that is, the maximal existence time $t^{+}(u_0) = \infty$ for
all $u_0 \in L^2(\Omega)$. Moreover, the modified system agrees with the
original system \eqref{eq:absSemiLinInvar} inside 
$B_{L^2(\Omega)}(0,\delta)$. Hence, the modification gives us a local
behaviour of the original system. 
% Although our assumptions imply that the original 
% system has a global solution for any initial data $u_0$, we modify the nonlinear term $f$
% in order to gain some useful estimates of the projections of semiflows in various directions (subspaces). 

In \cite{MR1000974}, invariant manifolds for the modified system are
constructed as follows. We choose the Lipschitz constant $\varepsilon$ of $\tilde f$
so that $\varepsilon < (\beta - \alpha)/4$ and there exists $\gamma$ such
that
\begin{equation}
 \label{eq:constGamma}
  - \beta + 2 \varepsilon < \gamma < - \alpha - 2 \varepsilon.
\end{equation} 
By abuse of notations, we denote again by $\Phi_t(u_0)$ the solution $u(t)$ of the modified system
with the initial condition $u_0$. 
Let 
\begin{displaymath}
 \begin{aligned}
 W^{-} &= \{ u \in L^2(\Omega) : e^{\gamma t} \Phi_{t}(u) \rightarrow 0
          \text{ as } t \rightarrow \infty \} \\
 W^{+} &= \{ u \in L^2(\Omega) : \text{ a backward branch } \Phi_t(u)
           \text{ exists for all } t \leq 0 \\
       &\quad \quad \text{ and } e^{\gamma t} \Phi_{t}(u) \rightarrow 0
           \text{ as } t \rightarrow -\infty \}. \\
\end{aligned}
\end{displaymath}
The main idea to show that $W^{-}$ and $W^{+}$ are invariant manifolds
is that certain cones and moving cones are positively invariant, which
can be determined by the difference in the growth rates on $X^{-}$ and $X^{+}$.
For $\lambda > 0$, we define a cone 
\begin{equation}
\label{eq:conKlambda}
 K_{\lambda} = \{(v,w) \in X^{-} \times X^{+} : \lambda \|v\|_{X^{-}}
                 \leq  \|w\|_{X^{+}} \}.
\end{equation}
It is shown in \cite[Lemma 2.4]{MR1000974} that $K_{\lambda}$ is
positively invariant if $\lambda \in [\mu,\nu]$ where $\mu$ and $\nu$ are
positive parameters  with $\mu < 1 < \nu$ satisfying
\begin{equation}
\label{eq:constMuNu1} 
 \varepsilon < (\beta - \alpha)/(2+\nu+\mu^{-1}).
\end{equation} 
Indeed, $\mu$ and $\nu$ can be further restricted so that
\begin{equation}
\label{eq:constMuNu2}
  \varepsilon(1+ \mu^{-1}) - \beta < \gamma < - \varepsilon(1+\nu) -\alpha.
\end{equation}
The next two theorems give the existence of global stable and global unstable invariant manifolds 
for the modified system. For the sake of mathematical necessity 
(when investigating the dependence on the domains), we sketch the proofs here.
  
\begin{theorem}[{\cite[Theorem 2.1]{MR1000974}}]
 \label{th:globalStaMan}
 There exists a Lipschitz function $h^{-}:X^{-} \rightarrow X^{+}$ such
 that $W^{-} = \graph(h^{-})$ and $h^{-}(0) = 0$.
\end{theorem}
\begin{proof}[Sketch of the proof]
 Fix $v_0 \in X^{-}$ and let
 \begin{displaymath}
 B = \{ w_0 \in X^{+} : \|w_0\|_{X^{+}} \leq \mu \|v_0\|_{X^{-}} \}.
 \end{displaymath}
 We write $\Phi_t(u_0) = u(t)$ as $u(t) = v(t) \oplus w(t)$ where 
 $v(t) \in X^{-}$ and $w(t) \in X^{+}$. Define
 \begin{displaymath}
   G_t = \{w_0 \in B : \|w(t)\|_{X^{+}} \leq \mu \|v(t)\|_{X^{-}} \}.
 \end{displaymath} 
It can be shown that $G_{\infty} := \bigcap_{t \geq 0} G_t$
contains exactly one element. 
A function $h^{-}$ defined by
%\begin{displaymath}
      $h^{-}(v_0) =  G_{\infty}$
%\end{displaymath}
for $v_0 \in X^{-}$ is a Lipschitz function with $h^{-}(0) = 0$ and $\graph(h^{-}) = W^{-}$.
\end{proof}
\begin{theorem}[{\cite[Theorem 2.2]{MR1000974}}]
\label{th:globalUnstaMan}
 There exists a Lipschitz function $h^{+}:X^{+} \rightarrow X^{-}$ such
 that $W^{+} = \graph(h^{+})$ and $h^{+}(0) = 0$.
\end{theorem}
\begin{proof}[Sketch of the proof]
The proof is based on a standard contraction mapping argument.
Let 
\begin{displaymath}
 Y = \{ h \in C(X^{+}, X^{-}) : h(0) = 0 \text{ and } h \text{ is } \nu^{-1} \text{-Lipschitz } \}.
\end{displaymath}
Then $Y$ is a complete metric space with the norm
\begin{equation}
\label{eq:LipsNormY}
  \|h\|_{\text{Lip}} = \sup_{w \ne 0} \frac{\|h(w)\|_{X^{-}}}{\|w\|_{X^{+}}}.
\end{equation}
For an arbitrary $h \in Y$, it can be shown that 
$P^{+} \Phi_t(\graph(h)) = X^{+}$ and that $\Phi_t(\graph(h))$ is the graph of
a $\nu^{-1}$-Lipschitz function for all $t \geq 0$. Hence, the map
$T_t : Y \rightarrow Y$  for $t \geq 0$ given by
\begin{displaymath}
  T_t(h) = \tilde{h}
\end{displaymath} 
where $\tilde{h} \in Y$ with $\graph(\tilde{h}) =\Phi_t(\graph(h))$ is well-defined.
Furthermore, $T_t$ is a contraction on $Y$ for $t$ sufficiently large. Indeed, 
\begin{displaymath}
\|T_t(h_2) - T_t(h_1) \|_{\text{Lip}} 
\leq \nu (\nu - \mu )^{-1} \exp{((\alpha - \beta + \varepsilon (2+\mu+\nu^{-1}))t)}
      \| h_2 - h_1 \|_{\text{Lip}}. 
\end{displaymath}
Hence, there exists a unique fixed point $h_t \in Y$ for $t$ sufficiently large. 
We can show that $h_t$ is a fixed point of $T_{\tau}$ for all $\tau \geq 0$ and
$h^{+} := h_t$ is the required Lipschitz  function with  $\graph(h^{+}) = W^{+}$ 
and $h^{+}(0) = 0$.
\end{proof}
\begin{remark}
\label{rem:tangencyHplus}
Let $ Y_0 = \{ h \in Y : h \text{ is differentiable at } 0 \text{ and } Dh(0) = 0 \}$.
Then $Y_0$ is closed in $Y$.  As $D \tilde f(0) = 0$ (in fact $Df(0)=0$ from Assumption
\ref{assump:nonlinearFInvar}),
it can be shown that $T_t : Y_0 \rightarrow Y_0$ for all 
$t > 0$. Hence, the fixed point $h^{+}$ in Theorem
\ref{th:globalUnstaMan} lies on $Y_0$
(see the proposition after the proof of Theorem 2.2 in \cite{MR1000974}). 
\end{remark}
The next two theorems give the existence of the local stable and the local unstable invariant manifolds
for \eqref{eq:absSemiLinInvar}.
\begin{theorem}[{\cite[Theorem 1.1(i)]{MR1000974}}]
\label{th:locStaMan}
 Under the assumptions given above, there exists an open neighbourhood
 $U$ of $0$ in $L^2(\Omega)$ such that $W^s$ is a Lipschitz manifold
 which is tangent to $X^{s}$ at $0$, that is, there exists a Lipschitz
 function $h^s: P^s(U) \rightarrow X^{cu}$ such that
 $\text{\upshape{graph}}(h^s) = W^s$, $h^s(0) =0$ and $h^s$ is
 differentiable at $0$ with $Dh^s(0)=0$.
\end{theorem}
\begin{proof}[Sketch of the proof] 
Set $X^{-} =X^{s}$ and  $X^{+} = X^{cu}$.
We take $\alpha = -\sigma$ and fix  $\beta$ such that $-\sigma < \beta < 0$. 
Renorm $X^{-}$ and $X^{+}$ by \eqref{eq:normXminusplus}. 
%and \eqref{eq:normXplus},respectively. 
By Assumption \ref{assump:nonlinearFInvar},
there exists $\delta > 0$ such that the modification $\tilde f$ has a Lipschitz constant 
$\varepsilon < (\beta - \alpha)/4$ and the modified system agrees with the original 
system on $B_{L^2(\Omega)} (0, \delta)$. By applying Theorem
\ref{th:globalStaMan}, we can find a product neighbourhood
$U \subset B_{L^2(\Omega)} (0, \delta)$ and prove that
$W^s = W^{-} \cap U$ is a local stable invariant manifold. It can be
shown that any local stable manifold constructed using another renorming
and modification agrees on a neighbourhood on which the manifolds are both defined.
The tangency condition $Dh^s(0)=0$ follows by making $\mu \rightarrow 0$ (by
letting $\varepsilon \rightarrow 0$ and possibly shrinking $U$).
\end{proof}
\begin{theorem}[{\cite[Theorem 1.2(i)]{MR1000974}}]
\label{th:locUnstaMan}
 Under the assumptions given above, there exists an open neighbourhood
 $U$ of $0$ in $L^2(\Omega)$ such that $W^u$ is a Lipschitz manifold
 which is tangent to $X^{u}$ at $0$, that is, there exists a Lipschitz
 function $h^u: P^u(U) \rightarrow X^{cs}$ such that
 $\text{\upshape{graph}}(h^u) = W^u$, $h^u(0) =0$ and $h^u$ is
 differentiable at $0$ with $Dh^u(0)=0$.
\end{theorem}
\begin{proof}[Sketch of the proof] 
Set $X^{-} =X^{cs}$ and  $X^{+} = X^{u}$. We take $\beta > 0$ such that
$\beta < \min\{\text{Re}(\lambda) : \lambda \in \sigma^u \}$ and fix
$\alpha$ such that $0 < \alpha < \beta$.
Renorm $X^{-}$ and $X^{+}$ and modify the nonlinearity $f$ as in the
proof of Theorem \ref{th:locStaMan}.
Applying Theorem \ref{th:globalUnstaMan}, we can find a product neighbourhood
$U \subset B_{L^2(\Omega)} (0, \delta)$ and prove that
$W^u = W^{+} \cap U$ is a local unstable invariant manifold. It can be
shown that any local unstable manifold constructed using another renorming
and modification agrees on a neighbourhood on which the manifolds are both defined.
The tangency condition $Dh^s(0)=0$ follows from Remark \ref{rem:tangencyHplus}.
\end{proof}
The product neigbourhood $U$ in Theorem \ref{th:locStaMan} and Theorem \ref{th:locUnstaMan}
can be chosen to be  $U = V_1 \times V_2$ where $V_1 \subset X^{-}$ is a ball of radius $\delta_1$ 
and $V_2 \subset X^{+}$ is a ball of radius $\delta_2$ such that 
$\delta_1 < \delta_2$ for the local stable manifold and $\delta_1 > \delta_2$ 
for the local unstable manifold. In fact, with these choices of product neighbourhoods, 
$W^s$ is positively invariant and $W^u$ is negatively invariant (
see property (P4) in \cite{MR1000974}).

\subsection{Existence of invariant manifolds for the perturbed equations}
\label{subsec:existInvPert}

In this section, we apply the construction of invariant manifold in Section
\ref{subsec:constructInv} to obtain invariant manifolds for the perturbed equations 
\eqref{eq:absSemiLinInvarN} under the assumptions stated in Theorem \ref{th:mainResults}
and Theorem \ref{th:mainResults2}. We first collect some preliminary results
on domain perturbation for solutions of parabolic equations and the corresponding elliptic
equations.

Under Mosco convergence (Assumption \ref{assum:bddPertMosco}) and the uniform boundedness
of the domains, it is known that if $\lambda \in \rho(-A)$, then $\lambda \in
\rho(-A_n)$ for $n$ sufficiently large and $(\lambda + A_n)^{-1} \rightarrow 
(\lambda + A)^{-1}$ in $\mathscr L(L^2(D))$ (see \cite[Corollary 4.7]{MR1955096}).
An important consequence is stated in the following lemma.
%
% $(\lambda + A_n)^{-1} f \rightarrow (\lambda + A)^{-1} f$
% in $L^2(D)$ as $n \rightarrow \infty$ for all $f \in L^2(D)$ if $\lambda \geq \lambda_0$
% (see \cite[Theorem 5.2.4]{Daners20081}).
% 
% the pseudo resolvents converge in the strong operator topology as stated below.
% \begin{theorem}[{\cite[Theorem 5.2.4]{Daners20081}}]
% \label{th:convPseudoResolStrongPW}
% If $\lambda \geq \hat \lambda_0$ then
%   $\mathcal R_n (\lambda)f \rightarrow \mathcal R(\lambda) f$  
% in $L^2(D)$ as $n \rightarrow \infty$ for all $f \in L^2(D)$.
% \end{theorem} 
%(see \cite[Theorem 5.2.4]{Daners20081}). This means $\mathcal R_n
% (\lambda)\rightarrow  \mathcal R (\lambda)$ in the strong
% operator topology.  
%
% Since $\Omega_n, \Omega \subset D$ for all
% $n \in \mathbb N$, the pseudo resolvents are compact on $L^2(D)$.
% We indeed have the convergence of pseudo resolvents in the operator norm. 
% \begin{theorem}[{\cite[Corollary 4.7]{MR1955096}}]
% \label{th:convPseudoResolOPnorm}   
%  If $\lambda \in \rho(-A)$ then $\lambda \in
% \rho(-A_n)$ for $n$ sufficiently large and  $\mathcal R_n
% (\lambda)\rightarrow  \mathcal R (\lambda)$ in $\mathscr L(L^2(D))$. 
% \end{theorem}
%

\begin{lemma}[{\cite[Corollary 4.2]{MR1955096}}]
\label{lem:convSpectProj}
Suppose that Assumption \ref{assum:An} and \ref{assum:bddPertMosco}
are satisfied.
If $\Sigma \subset \sigma(-A)$ is a compact spectral set and
$\Gamma$ is a rectifiable closed curve enclosing $\Sigma$ and separating
it from the remaining of spectrum, then $\sigma(-A_n)$ is separated by
$\Gamma$ into a compact spectral set $\Sigma_n$ and the rest of
spectrum for $n$ sufficiently large. Moreover, for the corresponding 
spectral projections $P$ and $P_n$, we have that the images of $P$ and $P_n$
have the same dimension and $P_n$ converges to $P$ in norm 
\end{lemma} 

We next consider the behaviour of solutions of the initial value problem 
\eqref{eq:absSemiLinInvar} under domain perturbation. Recall from Remark 
\ref{rem:conditionOnSubOp} (ii) that Assumption
\ref{assump:nonlinearFInvar} means $f$ is linear bounded with respect to $u$ and
consequently the solution of \eqref{eq:absSemiLinInvar} exists globally
for any initial condition $u_0 \in L^2(\Omega)$. We can state the convergence of solutions
of parabolic equations under domain perturbation in terms of semiflows as follows.

\begin{theorem} 
\label{th:convSolSemiLinParaDir}
Suppose that Assumption \ref{assum:An}, \ref{assum:bddPertMosco} and 
\ref{assump:nonlinearFInvar} are satisfied.
Let $u_{0,n} \in L^2(\Omega_n)$ and $u_{0} \in L^2(\Omega)$.
If $u_{0,n} |_{\Omega} \rightharpoonup u_0$ weakly in $L^2(\Omega)$, then
\begin{equation}
 \label{eq:convSolUsingSemiflow}
   \Phi_{t,n} (u_{0,n}) \rightarrow  \Phi_t(u_0)
\end{equation}
in $L^2(D)$ as $n \rightarrow \infty$ uniformly with respect to $t \in
(0,t_0]$ for all $t_0 \in (0, \infty)$.
Moreover, if $u_{0,n} \rightarrow  u_0$ strongly in $L^2(D)$, then
\eqref{eq:convSolUsingSemiflow} holds uniformly with respect to 
$t \in [0,t_0]$ for all $t_0 \in (0, \infty)$.
\end{theorem}
\begin{proof}
The assertion follows from similar arguments for the proof of \cite[Theorem 6.5]{MR2119988}
(the case of $-\Delta$), 
that is, by applying \cite[Theorem 4.6]{MR2119988}. The only minor modification is that we 
need to rescale the elliptic forms $a_n(\cdot,\cdot)$ and $a(\cdot,\cdot)$ into coercive forms
in order to apply \cite[Theorem 5.2]{MR1832168} to obtain the convergence of (degenerate) 
semigroups from the strong convergence of the resolvents. Note also that the convergence
result under stronger assumptions on domains can be found in \cite{MR1404388}. 
\end{proof}

To construct invariant manifolds for the perturbed problem \eqref{eq:absSemiLinInvarN},
we decompose $\sigma(-A_n) = \sigma_n^{s} \cup \sigma_n^c \cup \sigma_n^u$ where
$\sigma_n^{s}, \sigma_n^c$ and $\sigma_n^u$ are sets defined similarly to 
\eqref{eq:sigmaSCU}.
 By Lemma \ref{lem:convSpectProj}, we have that $\Gamma^c$ and $\Gamma^u$ 
separate $\sigma_n^c$ and $\sigma_n^u$ respectively from the remaining of spectrum for $n$ 
sufficiently large. The hyperbolicity assumption ($\sigma^c = \emptyset$) implies that
 $\sigma_n^c = \emptyset$  and hence the equilibrium $0 \in L^2(\Omega_n)$ of 
\eqref{eq:absSemiLinInvarN} is hyperbolic for all $n$ sufficiently large. 
We define the spectral projections $P_n^c$ and $P_n^u$ similarly to \eqref{eq:spectralProjs}
and write $P_n^s := 1 - P_n^c - P_n^u$. Note that the hyperbolicity assumption implies
$P_n^c = 0$ for $n$ sufficiently large.
% \begin{displaymath}
%  \begin{aligned}
%  P_n^c  &= \frac{1}{2\pi i} \int_{\Gamma^c} (\lambda + A_n)^{-1} d\lambda = 0\\
%  P_n^u  &= \frac{1}{2\pi i} \int_{\Gamma^u} (\lambda + A_n)^{-1} d\lambda \\
%  P_n^s  &= 1- P_n^c -P_n^u.
%  \end{aligned}
% \end{displaymath}
In addition, Lemma \ref{lem:convSpectProj} implies that 
\begin{equation}
\label{eq:convSpectalProjCU}
  P_n^c  \rightarrow  P^c  \quad \text{ and } \quad  P_n^u \rightarrow P^u  
\end{equation}
in $\mathscr L(L^2(D))$ as $n \rightarrow \infty$. 
We decompose 
\begin{equation}
\label{eq:splitL2OmegaN}
   L^2(\Omega_n) = X_n^s \oplus X_n^c \oplus X_n^u, 
\end{equation}
where $X_n^s$, $X_n^c$ and $X_n^u$ are the images of $P_n^s$, $P_n^c$
and $P_n^u$, respectively.
From the above consideration we have that  $X_n^c = \{ 0\}$ and $X_n^u$ is a 
finite dimensional subspace with $\dim(X_n^u) =\dim(X^u)$ for all $n$ sufficiently large. 
 
In order to obtain invariant manifolds for the modified system
of the perturbed equation \eqref{eq:absSemiLinInvarN}, we decompose $L^2(\Omega_n)$ as 
$X_n^{-} \oplus X_n^{+}$ in two different ways as in Section \ref{subsec:constructInv}.
In particular, $\dim(X_n^{+}) = \dim(X^{+}) < \infty$ for $n$ sufficiently large and
\begin{equation}
 \label{eq:convProplus}
     P_n^{+} \rightarrow  P^{+} 
\end{equation}
in $\mathscr L(L^2(D))$. 
% Moreover, if $u = v \oplus w \in
% L^2(\Omega_n)$ where $v \in X_n^{-}$ and  $w \in X_n^{+}$ then
% \begin{equation}
%  \label{eq:equivNormL2N}
%    \frac{1}{\|P_n^{-} \| +\|P_n^{+}\|} (\|v\|_{L^2(\Omega_n)} + \|w\|_{L^2(\Omega_n)})
%    \leq \|u\|_{L^2(\Omega_n)} 
%    \leq (\|v\|_{L^2(\Omega_n)} + \|w\|_{L^2(\Omega_n)}).
% \end{equation}   
By Assumption \ref{assum:An}, we can choose the parameters $\alpha$ and $\beta$
for the restriction of semigroup $S_n(t)$ to $X_n^{-}$ and $X_n^{+}$
uniformly with respect to $n \in \mathbb N$ so that $S_n^{-}(t)$ and
$S_n^{+}(t)$ satisfy similar estimates as in \eqref{eq:alphaSminus} and
\eqref{eq:betaSplus}, respectively.  
%Hence the renorming of $X_n^{-}$ and $X_n^{+}$ by
We can renorm $X_n^{-}$ and $X_n^{+}$ using similar norms involving 
$S_n^{-}(t)$ and $S_n^{+}(t)$ as defined in \eqref{eq:normXminusplus}.
% and \eqref{eq:normXplus},respectively.
% \begin{equation}
% \label{eq:normXminusN}
%   \|v\|_{X_n^{-}} := \sup_{t \geq 0} e^{-\alpha t} \|S_n^{-}(t) v\|_{L^2(\Omega_n)},
% \end{equation}
% for $v \in X_n^{-}$ and 
% \begin{equation} 
% \label{eq:normXplusN}
%   \|w\|_{X_n^{+}} := \sup_{t \leq 0} e^{-\beta t} \|S_n^{+}(t) w\|_{L^2(\Omega_n)},
% \end{equation}
%for $w \in X_n^{+}$ give equivalent norms on $X_n^{-}$ and $X_n^{+}$, respectively. 
%Indeed, we have
% \begin{equation}
%  \label{eq:equivNormXminusN}
%    \|v\|_{L^2(\Omega_n)} \leq \|v\|_{X_n^{-}} \leq M_1 \|v\|_{L^2(\Omega_n)}
% \end{equation}
% for all $v \in X_n^{-}$ and 
% \begin{equation}
%  \label{eq:equivNormXplusN}
%    \|w\|_{L^2(\Omega_n)} \leq \|w\|_{X_n^{+}} \leq M_2 \|w\|_{L^2(\Omega_n)}
% \end{equation}
% for all $w \in X_n^{+}$ where $M_1$ and $M_2$ can be chosen uniformly with
% respect to $n \in \mathbb N$.
In particular, similar estimates as in \eqref{eq:equivNormXminusplus} 
%and \eqref{eq:equivNormXplus} 
hold for the norms $\| \cdot \|_{X_n^{-}}$ and 
$\| \cdot \|_{X_n^{+}}$ with uniform constants $M_1$ and $M_2$ for $n$ sufficiently large.   

By Assumption \ref{assump:nonlinearFInvar}  (ii), 
there exists $\delta > 0$ independent of $n$ such that the modification $\tilde f_n$ of $f_n$ has
a Lipschitz constant $\varepsilon < (\beta - \alpha)/4$  and the modified system 
agrees with the original system on $B_n := B_{L^2(\Omega_n)}(0, \delta)$
for all $n \in \mathbb N$. Therefore, we can construct 
the stable and unstable invariant manifold for the modified system  
by using uniform parameters $\gamma,\mu$ and $\nu$ for all $n$ large. 
By Theorem \ref{th:locStaMan}, there exists a product neighbourhood $U_n \subset B_n$ 
such that a local stable invariant manifold is $W_n^{s}(U_n) = \graph(h_n^{-}) \cap U_n$.
Since the parameters $\alpha$ and $\beta$ are chosen uniformly for the renorming of 
$X_n^{-}$ and $X_n^{+}$ respectively, we can choose $U_n \subset B_n$ to be a product neighbourhood
$V_{1,n} \times V_{2,n}$ where $V_{1,n} \subset X_n^{-}$ is a ball of radius $\delta_1$ 
and $V_{2,n} \subset X_n^{+}$ is a ball of radius $\delta_2$ with $\delta_1 < \delta_2$ 
for all $n \in \mathbb N$. Without loss of generality we may choose $\delta$ smaller so that 
the modified system agrees with the original system on $\overline{B}_n$
for all $n \in \mathbb N$.
Similarly, by Theorem \ref{th:locUnstaMan}, there exists a product neighbourhood $\tilde{U}_n \subset B_n$ 
such that a local unstable invariant manifold is $W_n^{u}(\tilde{U}_n) = \graph(h_n^{+}) \cap \tilde{U}_n$.
Since the parameters $\alpha$ and $\beta$ are chosen uniformly for the renorming of 
$X_n^{-}$ and $X_n^{+}$ respectively, we can choose $\tilde{U}_n \subset B_n$ to be a product neighbourhood
$\tilde V_{1,n} \times \tilde V_{2,n}$ where $\tilde V_{1,n} \subset X_n^{-}$ is a ball of radius $\tilde \delta_1$ 
and $\tilde V_{2,n} \subset X_n^{+}$ is a ball of radius $\tilde \delta_2$ with $\tilde \delta_1 > \tilde \delta_2$ 
for all $n \in \mathbb N$. Again we may choose $\delta$ smaller so that 
the modified system agrees with the original system on $\overline{B}_n$ for all $n \in \mathbb N$.
Therefore, we have established the existence of  local unstable manifolds and local stable manifolds for
the perturbed problem \eqref{eq:absSemiLinInvarN}.

%A similar consideration applies to the problem on $\Omega$.  

We can assume that the choice of neighbourhoods considered above applies to the limit problem 
\eqref{eq:absSemiLinInvar} (by possibly shrinking $\delta$).
Therefore, to prove Theorem \ref{th:mainResults} and Theorem \ref{th:mainResults2},
 it remains to verify the continuity under domain perturbation (upper and lower
semicontinuity) of local stable and local unstable invariant manifolds
inside some ball $B_n = B_{L^2(\Omega_n)}(0,\hat \delta)$ contained in $U_n = V_{1,n} \times V_{2,n}$
or $\tilde U_n = \tilde V_{1,n} \times \tilde V_{2,n}$.
\begin{remark}
By our assumptions and the application of \cite[Theorem 5.2]{MR1832168}, we know that
$S_n(t)$ converges to $S(t)$ in the strong operator topology uniformly with respect to $t$
on \emph{compact subsets} of $(0,\infty)$. The main difficulty to prove 
upper and lower semicontinuity of invariant manifolds using the construction in 
\cite{MR1000974} is that we need to deal with sequences of functions under a sequence of the
special norms $\| \cdot \|_{X_n^{-}}$ and $\| \cdot \|_{X_n^{+}}$ defined in terms of the 
supremum of $e^{-\alpha t} \|S_n^{-}(t) v\|_{L^2(\Omega_n)}$ on a non-compact interval
$[0, \infty)$ and the supremum of $e^{-\beta t} \|S_n^{+}(t) w\|_{L^2(\Omega_n)}$ on 
$(-\infty, 0]$, respectively (see \eqref{eq:normXminusplus}).
% and \eqref{eq:normXplus}). 
In particular, we do not generally have the convergence of a sequence of functions in $X_n^{-}$
or $X_n^{+}$ with respect to a sequence of the norms mentioned above.  
\end{remark}

\section{Some technical results towards the proof of semicontinuity}
\label{sec:technicalLemma}
\sectionmark{Some technical results}
In this section, we give some technical results required to prove upper and lower
semicontinuity in Theorem \ref{th:mainResults} and Theorem \ref{th:mainResults2}.
In particular, we prove some convergence result for a
bounded sequence $(w_n)_{n \in \mathbb N}$ with $w_n \in X_n^{+}$ for
each $n \in \mathbb N$. Moreover, we give a characterization of upper and
lower semicontinuity.

\subsection{Convergence of sequences in finite dimensional subspaces}
% We first give a direct consequence of convergence of the spectral projections in \eqref{eq:convSpectalProjCU}.
\begin{lemma}
\label{lem:convUsUcUu}
Let $(\phi_n)_{n \in \mathbb N}$ be a sequence with $\phi_n \in L^2(\Omega_n)$ 
for each $n \in \mathbb N$ and $\phi \in L^2(\Omega)$. 
We decompose $\phi_n := \phi_n^s \oplus \phi_n^c \oplus \phi_n^u$
corresponding to the decomposition \eqref{eq:splitL2OmegaN}.
Similarly, we decompose $\phi := \phi^s \oplus \phi^c \oplus \phi^u$.
%\begin{itemize}
%\item[\upshape{(i)}]
 If $\phi_n \rightarrow  \phi$ strongly in $L^2(D)$,
 then $\phi_n^{*} \rightarrow \phi^{*}$ strongly in $L^2(D)$ for $* =s,c,u$. 
% \item[\upshape{(ii)}]
%  If $i_n \phi_n \rightharpoonup i \phi$ in $L^2(D)$ weakly,
%  then $i_n \phi_n^{*} \rightharpoonup i \phi^{*}$ in $L^2(D)$ weakly for $* =s,c,u$.
%  \item[\upshape{(iii)}]
%  If $r i_n \phi_n \rightarrow \phi$ in $L^2(\Omega)$ strongly
%  and $\| i_n \phi_n\|_{L^2(D)}$ is uniformly bounded,
%  then $r i_n \phi_n^{*} \rightarrow  \phi^{*}$ in $L^2(\Omega)$  strongly for $* =s,c,u$.
%  \item[\upshape{(iv)}]
%  If $r i_n \phi_n \rightharpoonup \phi$ in $L^2(\Omega)$ weakly
%  and $\| i_n \phi_n\|_{L^2(D)}$ is uniformly bounded,
%  then  \linebreak $r i_n \phi_n^{*} \rightharpoonup  \phi^{*}$ in $L^2(\Omega)$ weakly for $* =s,c,u$.
% \end{itemize}
\end{lemma}
\begin{proof}
%Suppose that $ \phi_n \rightarrow  \phi$ in $L^2(D)$ strongly.
A direct application of \eqref{eq:convSpectalProjCU} implies 
%
% Originally I use this one. But adding (iii) and (iv) need the other one.
%
% \begin{displaymath}
% \begin{aligned}
% \|i_n \phi_n^{*} -i \phi^{*}\|_{L^2(D)} 
% &= \| i_n P_n^{*} r_n i_n \phi_n -i_n P_n^{*}r_n i \phi\|_{L^2(D)}\\
%   &\quad +\| i_n P_n^{*} r_n i \phi -i P^{*}r i \phi \|_{L^2(D)}\\
% &\leq \| i_n P_n^{*} r_n \| \| i_n \phi_n -i \phi\|_{L^2(D)}\\
%   &\quad +\| i_n P_n^{*} r_n  -i P^{*}r \| \| i \phi \|_{L^2(D)} \\
% &\rightarrow 0
% \end{aligned}
% \end{displaymath}
%
% \begin{displaymath}
% \begin{aligned}
% \| \phi_n^{*} - \phi^{*}\|_{L^2(D)} 
% &\leq \| P_n^{*}  \phi_n - P^{*} \phi_n |_{\Omega} \|_{L^2(D)}
%   +\| P^{*}  \phi_n |_{\Omega}  - P^{*} \phi \|_{L^2(D)}\\
% &\leq \| P_n^{*} - P^{*} \| \| \phi_n \|_{L^2(D)}
%    +\|  P^{*} \| \| \phi_n -  \phi \|_{L^2(D)} \\
% &\rightarrow 0
% \end{aligned}
% \end{displaymath}
% as $n \rightarrow \infty$ for $* =c,u$. 
$\phi_n^c \rightarrow \phi^c$ and $\phi_n^u \rightarrow \phi^u$ in $L^2(D)$.
Since $\phi_n^s = (1-P_n^c -P_n^u) \phi_n$ 
and $\phi^s = (1-P^c -P^u) \phi$, 
we also get $ \phi_n^s \rightarrow  \phi^s$ in $L^2(D)$. 
% This proves assertion (i).
% The statement (iii) follows from a similar argument above. 
%
% Suppose now that $i_n \phi_n \rightharpoonup i \phi$ in $L^2(D)$ weakly.
% By \eqref{eq:convSpectalProjCU} and the weak convergence of $i_n \phi_n$, 
% we have for $\chi \in L^2(D)$ and $* = c,u$ 
% \begin{displaymath}
%  \begin{aligned}
% \Big |(i_n \phi_n^{*}) - i \phi^{*})| \chi)_{L^2(D)} \Big|
% &= \left |\left ( i_n P_n^{*}  \phi_n - i P^{*} \phi | \chi \right )_{L^2(D)}\right |\\
% & \leq \left | \left ( i_n P_n^{*} \phi_n - i P^{*}r i_n \phi_n |\chi \right )_{L^2(D)} \right |\\
% &\quad  + \left |\left (i P^{*}r i_n \phi_n - i P^{*} \phi | \chi \right )_{L^2(D)} \right |\\
% &\leq \|i_n P_n^{*} r_n - i P^{*} r\| \|i_n \phi_n\|_{L^2(D)} \|\chi\|_{L^2(D)}\\
% &\quad + \left |\left ( i_n \phi_n -i \phi | (i P^{*} r)' \chi \right)_{L^2(D)} \right | \\ 
% &\rightarrow 0 
% \end{aligned}
% \end{displaymath}
%  as $n \rightarrow \infty$ where $(i P^{*} r)'$ denotes the adjoint operator of $i P^{*} r$. 
% This implies $i_n(x_n^{*}) \rightharpoonup i(x^{*})$ in $L^2(D)$ weakly for $* = c,u$. 
% Since $x_n^s = (1-P_n^c -P_n^u) x_n$  and $x^s = (1-P^c -P^u) x$, 
% we obtain the weak convergence $i_n(x_n^s) \rightharpoonup i(x^s)$  in $L^2(D)$.
% This proves assertion (ii). The statement (iv) can be proved in a similar fashion.   
\end{proof}
%
%
%
% Note that if $i_n \phi_n$ is uniformly bounded
% in $L^2(D)$, then there exists a weak convergent subsequence in $L^2(D)$. 
% In general, we do not know whether the weak limit is zero almost everywhere on $D \backslash \Omega$,
% that is of the form $i \phi$ for some $\phi \in L^2(\Omega)$. This is the reason we include 
% both assertions (ii) and (iv) in Lemma \ref{lem:convUsUcUu} above. 
\begin{remark}
\label{remark:NonConvPminus}
The convergence $\phi_n^s \rightarrow \phi^s$ in Lemma \ref{lem:convUsUcUu} 
is different to convergence of the projections  
 $(1-P_n^c-P_n^u)  \rightarrow  (1-P^c-P^u)$ in $\mathscr L(L^2(D))$. 
 For example, consider a square domain $\Omega$ in $\mathbb R^2$ 
 perturbed by attaching ``fingers'' to one of the sides. 
 If we increase the number of fingers so that the
 measure remains the same (by letting their width go to zero). Then
 $|\Omega_n \backslash \Omega|$ is a positive constant for all $n \in
 \mathbb N$. It is known that $H^1_0(\Omega_n)$ converges to $H^1_0(\Omega)$ in the sense of Mosco
 (see \cite[Example 8.4]{MR1955096}). 
 Let $f \in L^2(D)$ be the constant function $1$. 
 By \eqref{eq:convSpectalProjCU}, we have that
 $P_n^c f \rightarrow  P^c f$ and $P_n^u f
\rightarrow P^u f$ in $L^2(D)$. If 
$(1- P_n^c -P_n^u) f \rightarrow (1-P^c-P^u) f$ in $L^2(D)$,
then $ f|_{\Omega_n} \rightarrow f|_{\Omega}$ in $L^2(D)$. This cannot be true
because $\|f|_{\Omega_n} - f|_{\Omega}\|_{L^2(D)} = |\Omega_n \backslash \Omega| >
0$ for all $n \in \mathbb N$. Hence, $(1-P_n^c-P_n^u)$ 
does not converge to $(1-P^c-P^u)$ in $\mathscr L(L^2(D))$.
Note that if we impose the assumption that the Lebesgue measure of the domain
converges, that is, $|\Omega_n| \rightarrow |\Omega|$ as $n \rightarrow
\infty$, then we obtain the convergence 
 $(1-P_n^c-P_n^u) \rightarrow (1-P^c-P^u)$ in  $\mathscr L(L^2(D))$.  
\end{remark}

In the next few results, we consider an arbitrary finite dimensional
subspace of $L^2(\Omega_n)$.

\begin{lemma}
\label{lem:equivNormCn}
Let $m$ be a positive integer. Suppose $V_n$ is an $m$-dimensional subspace
of $L^2(\Omega_n)$ with a basis $\{ f_{1,n}, f_{2,n}, \ldots f_{m,n}\}$ 
for each $n \in \mathbb N$, and $V$ is an $m$-dimensional 
subspace of $L^2(\Omega)$ with a basis $\{ f_{1}, f_{2}, \ldots f_{m}\}$.
If $f_{j,n} \rightarrow f_j$ in $L^2(D)$ as $n \rightarrow \infty$ 
for all $j = 1, \ldots, m$, then there exists $\hat c > 0$ such that 
\begin{displaymath}
c_n := \inf \Big \{ \Big \| \sum_{j=1}^m \xi_j f_{j,n}  \Big\|_{L^2(\Omega_n)} :
              \xi =(\xi_1, \ldots, \xi_m) \in \mathbb R^m , |\xi| =1 \Big \}
     \geq \hat c,
\end{displaymath}
for all $n \in \mathbb N$.
\end{lemma}
\begin{proof}
Let $\xi = (\xi_1, \ldots, \xi_m) \in \mathbb R^m$ with $|\xi| =1$. 
By convergence of the bases, we get
\begin{displaymath}
 \begin{aligned}
 \Big \| \sum_{j=1}^m \xi_j f_{j,n} - \sum_{j=1}^m \xi_j f_{j} \Big\|_{L^2(D)}
 &\leq \sum_{j=1}^m |\xi_j| \| f_{j,n} - f_{j} \|_{L^2(D)}  \\
 &\leq \sum_{j=1}^m  \| f_{j,n} - f_{j} \|_{L^2(D)}  \\
 &\rightarrow 0
 \end{aligned}
\end{displaymath}
as $n \rightarrow \infty$. Notice that the above convergence does not depend on
$\xi$. This means 
$ \| \sum_{j=1}^m \xi_j f_{j,n}  \|_{L^2(\Omega_n)} \rightarrow
 \| \sum_{j=1}^m \xi_j f_{j} \|_{L^2(\Omega)}$ uniformly with respect to
$\xi \in \mathbb R^m$ with $|\xi| =1$. 
Let
\begin{displaymath}
c:= \inf \Big \{ \Big \| \sum_{j=1}^m \xi_j f_{j}  \Big\|_{L^2(\Omega)} :
              \xi =(\xi_1, \ldots, \xi_m) \in \mathbb R^m , |\xi| =1 \Big \}.
\end{displaymath}
In particular, choosing $\zeta > 0$ such that $c-\zeta > 0$, 
there exists $N_0 \in \mathbb N$ 
(independent of $\xi \in \mathbb R^m$ with $|\xi| =1$) such that
\begin{displaymath}
 \Big\| \sum_{j=1}^m \xi_j f_{j,n}  \Big \|_{L^2(\Omega_n)} 
 \geq  \Big \| \sum_{j=1}^m \xi_j f_{j} \Big \|_{L^2(\Omega)} - \zeta,
\end{displaymath}
for all $n > N_0$ and for all $\xi \in \mathbb R^m$ with $|\xi| =1$.
Since $ \| \sum_{j=1}^m \xi_j f_{j}  \|_{L^2(\Omega)} \geq c$,
it follows that
%\begin{equation}
 $\left\| \sum_{j=1}^m \xi_j f_{j,n}  \right \|_{L^2(\Omega_n)} 
 \geq  c - \zeta$
%\end{equation}
for all $n > N_0$ and for all $\xi \in \mathbb R^m$ with $|\xi| =1$.
Taking the infimum  over $\xi \in \mathbb R^m$ with $|\xi| =1$, we obtain
%\begin{displaymath} 
   $c_n \geq c - \zeta > 0$
%\end{displaymath}
for all $n \geq N_0$. Finally, taking $\hat c := \min\{c_1,\ldots c_{N_0},  c- \zeta \}$,
the lemma follows.
\end{proof}
An immediate application of Lemma \ref{lem:equivNormCn} is the following result.
\begin{corollary}
\label{cor:convFuncMdim}
Assume that $V_n$ and $V$ are as in Lemma \ref{lem:equivNormCn} and that the 
convergence of bases $f_{j,n} \rightarrow f_j$ in $L^2(D)$ 
as $n \rightarrow \infty$ holds for all $j = 1, \ldots, m$. 
Let $u_n$ be a sequence such that $u_n \in V_n$ for each $n \in \mathbb N$.
If $\|u_n\|_{L^2(\Omega_n)}$ is uniformly bounded, then there exists 
a subsequence $u_{n_k}$ such that $u_{n_k} \rightarrow u$
in $L^2(D)$ with a limit $u \in V$.
\end{corollary}
\begin{proof}
For each $n \in \mathbb N$, we write $u_n = \sum_{j=1}^m \xi_{j,n} f_{j,n}$.
By a standard argument in the proof of equivalence of norms for finite dimensional spaces,
\begin{displaymath}
 \sum_{j=1}^m |\xi_{j,n}| \leq \frac{m}{c_n} \|u_n\|_{L^2(\Omega_n)},
\end{displaymath}
for all $n \in \mathbb N$, where $c_n$ is given in Lemma \ref{lem:equivNormCn}.
It follows from the uniform boundedness of $\|u_n\|_{L^2(\Omega_n)}$ and 
Lemma \ref{lem:equivNormCn} that $\sum_{j=1}^m |\xi_{j,n}|$ is uniform bounded.
We can extract a subsequence $\xi_{j,n_k}$ such that $\xi_{j,n_k} \rightarrow \xi_j$
for all $j = 1, \ldots, m$. Hence,
%\begin{displaymath}
   $u_{n_k} \rightarrow  u :=  \sum_{j=1}^m \xi_j f_j$ 
%\end{displaymath}
in $L^2(D)$. 
\end{proof}
Recall that we have $\dim(X_n^{+}) = \dim(X^{+}) < \infty$ for
sufficiently large $n$. We set $d:=\dim(X^{+})$ and fix a certain basis
$\{f_1,f_2, \ldots, f_d \}$ of $X^{+}$. Define
\begin{equation}
\label{eq:basisForXnPlus}
 f_{j,n} := P_n^{+} f_j|_{\Omega_n},
\end{equation}
for $j =1, \ldots, d$. Then we obtain a basis of $X_n^{+}$ as shown below.
\begin{theorem}
\label{th:basisXnPlus}
There exists $N_0 \in \mathbb N$ such that 
$\{f_{1,n},f_{2,n}, \ldots, f_{d,n} \}$ where $f_{j,n}$ defined by \eqref{eq:basisForXnPlus}
is a basis of $X^{+}_n$  for each $n > N_0$. Moreover, 
$f_{j,n} \rightarrow f_j$ in $L^2(D)$ 
as $n \rightarrow \infty$ holds for all $j = 1, \ldots, d$. 
\end{theorem} 
\begin{proof}
The convergence $f_{j,n} \rightarrow f_j$ is clear from
the definition of $f_{j,n}$ and \eqref{eq:convProplus}.  
Since $X^{+}_n$ is $d$-dimensional subspace for all $n$ sufficiently large,
it suffices to show that there exists $N_0 \in \mathbb N$ such that
$f_{1,n},f_{2,n}, \ldots, f_{d,n}$ are linearly independent for each $n > N_0$.
We prove this by using mathematical induction on $m$ for $m = 1,\ldots, d$ in the following 
statement: there exists $N_m \in \mathbb N$ such that
$f_{1,n},f_{2,n}, \ldots, f_{m,n}$ are linearly independent for each $n > N_m$.

The statement is trivial for $m=1$. For the induction step, suppose that the statement is true 
for $1, \ldots, m$ with $m<d$, but there is no $N_{m+1} \in \mathbb N$ such that
$f_{1,n},f_{2,n}, \ldots, f_{m+1,n}$ are linearly independent for each $n > N_{m+1}$.
%This implies that for each $n \in \mathbb N$, there exists $l \in \mathbb N$ with $l > n$ 
%such that $f_{1,l},f_{2,l}, \ldots, f_{m+1,l}$ are linearly dependent.
Thus, we can extract a subsequence $n_k$ (choosing $n_k > N_m$ for all $k \in \mathbb N$) 
such that $f_{1,n_k},f_{2,n_k}, \ldots, f_{m+1,n_k}$ are linearly dependent 
for all $k \in \mathbb N$.
By the linear independence of  $f_{1,n_k},f_{2,n_k}, \ldots, f_{m,n_k}$,
we can write
%\begin{displaymath}
  $f_{m+1,n_k} = \sum_{j=1}^m \xi_{j,n_k} f_{j,n_k}$
%\end{displaymath}
for all $k \in \mathbb N$. Since $f_{m+1,n_k} \rightarrow f_{m+1}$
in $L^2(D)$ as $k \rightarrow \infty$, 
it follows that $\|f_{m+1,n_k}\|_{L^2(\Omega_{n_k})}$ is uniformly bounded. 
Corollary \ref{cor:convFuncMdim} implies that there exists a subsequence denoted again by
$f_{m+1,n_k}$ such that $f_{m+1,n_k} \rightarrow f$ in $L^2(D)$
as $k \rightarrow \infty$, where the limit $f$ belongs to the $m$-dimensional 
subspace spanned by $f_1, f_2, \ldots, f_m$. By the uniqueness of a limit,
we conclude that $f_{m+1} = f$. This is a contradiction to the assumption that 
$\{f_1,f_2, \ldots, f_d \}$ is a basis of $X^{+}$. Hence, the induction 
statement is true for $m+1$ and the theorem is proved.
\end{proof}
As a consequence, we obtain the following convergence of a bounded sequence with
each term belongs to a sequence of the spaces $X_n^{+}$.
\begin{corollary}
 \label{cor:convBoundSeqInXplus}
 Let $(w_n)_{n \in \mathbb N}$ be a sequence with $w_n \in X_n^{+}$ for
each $n \in \mathbb N$. If $\|w_n\|_{L^2(\Omega_n)}$ (or
$\|w_n\|_{X_n^{+}}$) is uniformly bounded, then there exists a subsequence
$w_{n_k}$ such that 
%\begin{displaymath}
    $w_{n_k} \rightarrow w$
%\end{displaymath}
in $L^2(D)$ with the limit $w \in X^{+}$.
\end{corollary}
\begin{proof}
% Note that $\|w_n\|_{X_n^{+}} < C$ for all $n \in \mathbb N$ implies
% $\|w_n\|_{L^2(\Omega_n)} < C$ for all $n \in \mathbb N$ (\eqref{eq:equivNormXplus}).
The result follows immediately from Corollary \ref{cor:convFuncMdim} and
Theorem \ref{th:basisXnPlus} and the equivalence of norms on $X_n^{+}$. 
\end{proof}
\begin{remark}
The above result implies that there exists a subsequence $w_{n_k}$ such that
$\|w_{n_k}\|_{L^2(\Omega_{n_k})} \rightarrow \|w\|_{L^2(\Omega)}$ but does
not implies $\|w_{n_k}\|_{X_{n_k}^{+}} \rightarrow \|w\|_{X^{+}}$ as
degenerate semigroup only converges uniformly on compact subsets of $(0,\infty)$.
\end{remark}
%
%
%
% We note that the convergence of components in finite dimensional subspaces in
% Lemma \ref{lem:convUsUcUu} (ii) is indeed a strong convergence as stated below.
% \begin{lemma}
% \label{lem:convUsUsUuweak}
% Suppose that the assumptions in Lemma \ref{lem:convUsUcUu} (ii) are satisfied.
% If $i_n \phi_n \rightharpoonup i \phi$ in $L^2(D)$ weakly,
% then $i_n \phi_n^{s} \rightharpoonup i \phi^{s}$ in $L^2(D)$ weakly and
% $i_n \phi_n^{*} \rightarrow i \phi^{*}$ in $L^2(D)$ strongly for $* =c,u$.
% \end{lemma}
% \begin{proof}
% From Lemma \ref{lem:convUsUcUu}, we have a weak convergence 
% $i_n \phi_n^{*} \rightharpoonup i \phi^{*}$ in $L^2(D)$ for $* =s,c,u$.
% Since $X_n^u$ is a finite dimensional subspace, Corollary \ref{cor:convBoundSeqInXplus}
% implies that for any subsequence $n_k$  we can extract a further convergent subsequence 
% $i_{n_k} \phi_{n_k}^u \rightarrow \xi^u$ in $L^2(D)$ for some $\xi^u \in L^2(D)$. 
% The weak convergence above implies that $\xi^u = i \phi^u$ and hence 
% $i_{n_k} \phi_{n_k}^u \rightarrow i \phi^u$ in
% $L^2(D)$ strongly. Since this is true for any subsequence, we conclude that
% the whole sequence $i_n \phi_n^u$ converges to $i \phi^u$ in $L^2(D)$ strongly. 
% The same argument shows that $i_n(x_n^c) \rightarrow i(x^c)$ in $L^2(D)$.
% \end{proof}

%
%
%
\subsection{Characterisation of upper and lower semicontinuity} 
We give some equivalent statements for upper and lower 
semicontinuity mentioned in Theorem \ref{th:mainResults} and Theorem \ref{th:mainResults2}. 
We simplify the notations by considering bounded subsets
$W_n, W$ of $L^2(D)$. 

\begin{lemma}[Characterisation of upper semicontinuity]
 \label{lem:upperSemiContCharac}
The following statements are equivalent.
 \begin{itemize}
  \item[\upshape(i)] 
   %\begin{displaymath}
    $\sup_{v \in W_n} \inf_{u \in W} \| v-u \|_{L^2(D)} \rightarrow 0$
   %\end{displaymath}
   as $n \rightarrow \infty$.
 \item[\upshape(ii)] For any sequence $\{v_n\}_{n \in \mathbb N}$ with 
    $v_n \in W_n$, we have
    %\begin{displaymath}
     $\inf_{u \in W} \|v_n - u\|_{L^2(D)} \rightarrow 0$
    %\end{displaymath}
     as $n \rightarrow \infty$.  
 \item[\upshape(iii)] For any sequence $\{v_n\}_{n \in \mathbb N}$ with 
    $v_n \in W_n$, if $\{v_{n_k} \}_{k \in \mathbb N}$ is
    a subsequence, then there exist a further subsequence (denoted again by $v_{n_k}$)
    and  a sequence $\{u_{n_k}\}_{k \in \mathbb N}$ with $u_{n_k} \in W$ 
    such that
    $\|v_{n_k} - u_{n_k}\|_{L^2(D)} \rightarrow 0$  as $k \rightarrow \infty$.     
\end{itemize}
\end{lemma}
\begin{proof}
The statement (i) $\Rightarrow$ (ii) is clear.
%  as 
% \begin{displaymath}
%   \inf_{u \in W} \|v_n - u\|_{L^2(D)} \leq  \sup_{v \in W_n} \inf_{u \in W} \| v-u \|_{L^2(D)}
% \end{displaymath}
% for any sequence $\{v_n\}_{n \in \mathbb N}$ with $v_n \in W_n$.
%
For (ii) $\Rightarrow$ (i), we prove by contrapositive. Suppose that (i) fails.
Then 
\begin{displaymath}
\limsup_{n \rightarrow \infty} \left \{
     \sup_{v \in W_n} \inf_{u \in W} \| v-u \|_{L^2(D)} \right \} 
  =: a >0.
\end{displaymath}
%Hence, there exists 
We can find a subsequence $n_k \rightarrow \infty$ such that 
%\begin{displaymath}
  $\sup_{v \in W_{n_k}} \inf_{u \in W} \| v-u \|_{L^2(D)} \rightarrow a$
%\end{displaymath}
as $k \rightarrow \infty$. This implies that there exists  $v_{n_k} \in  W_{n_k}$
such that 
\begin{displaymath}
\inf_{u \in W} \| v_{n_k}-u \|_{L^2(D)} > a/2,
\end{displaymath} 
for all $k \in \mathbb N$. Hence, (ii) fails. 

For the statement (ii) $\Leftrightarrow$ (iii), notice first that 
$\inf_{u \in W} \|v_n - u\|_{L^2(D)} \rightarrow 0$ as $n \rightarrow \infty$
 if and only if there exists $u_n \in W$ such that 
$\|v_n -u_n \|_{L^2(D)} \rightarrow 0$ as $n \rightarrow \infty$.  
To see this, we choose $u_n \in W$ such that
$\|v_n - u_n\|_{L^2(D)} < \inf_{u \in W} \|v_n -u\|_{L^2(D)} + 1/n$ for each $n \in \mathbb N$.
Then the forward implication follows. 
The backward implication is clear as  $\inf_{u \in W} \|v_n -u\|_{L^2(D)} < \|v_n - u_n\|_{L^2(D)}$
for all $u_n \in W$. The statement (ii) $\Leftrightarrow$ (iii) 
then simply follows from the above and a standard subsequence characterisation of a limit. 
\end{proof}
By a similar argument, we can state the following lemma.
\begin{lemma}[Characterisation of lower semicontinuity]
\label{lem:lowerSemiContCharac} 
The following statements are equivalent.
 \begin{itemize}
  \item[\upshape(i)] 
   %\begin{displaymath}
    $\sup_{u \in W} \inf_{v \in W_n} \| v-u \|_{L^2(D)} \rightarrow 0$
   %\end{displaymath}
   as $n \rightarrow \infty$.
 \item[\upshape(ii)] For any sequence $\{u_n\}_{n \in \mathbb N}$ with 
    $u_n \in W$, we have
    %\begin{displaymath}
     $\inf_{v \in W_n} \|v - u_n\|_{L^2(D)} \rightarrow 0$
    %\end{displaymath}
     as $n \rightarrow \infty$.  
 \item[\upshape(iii)] For any sequence $\{u_n\}_{n \in \mathbb N}$ with 
    $u_n \in W$, if $\{u_{n_k} \}_{k \in \mathbb N}$ is
    a subsequence, then there exist a further subsequence (denoted again by $u_{n_k}$)
    and  a sequence $\{v_{n_k}\}_{k \in \mathbb N}$ with $v_{n_k} \in W_{n_k}$ 
    such that
    $\|v_{n_k} - u_{n_k}\|_{L^2(D)} \rightarrow 0$  as $k \rightarrow \infty$.     
\end{itemize}
\end{lemma}
%

% \begin{lemma}
% \label{lem:convUsUcUu}
% Let $(x_n)_{n \in \mathbb N}$ be a sequence in with $x_n \in L^2(\Omega_n)$. 
% Suppose $i_n(x_n) \rightarrow i(x)$ in $L^2(D)$ with the limit $x \in L^2(\Omega)$.
% If we decompose $x_n$ and $x$ as $x_n := x_n^s \oplus x_n^c \oplus x_n^u$ 
% and  $x := x^s \oplus x^c \oplus x^u$,
% then $i_n(x_n^{*}) \rightarrow i(x^{*})$ in $L^2(D)$ for $* =s,c,u$. 
% \end{lemma}
% \begin{proof}
% By \eqref{eq:convSpectalProjCU}, we have 
% \begin{displaymath}
% \begin{aligned}
% \|i_n (x_n^{*}) -i(x^{*})\|_{L^2(D)} 
% &= \| i_n P_n^{*} r_n i_n (u_n) -i_n P_n^{*}r_n i(u)\|_{L^2(D)}\\
%   &\quad +\| i_n P_n^{*} r_n i (u) -i P^{*}r i(u)\|_{L^2(D)}\\
% &\leq \| i_n P_n^{*} r_n \| \| i_n (u_n) -i(u)\|_{L^2(D)}\\
%   &\quad +\| i_n P_n^{*} r_n i (u) -i P^{*}r i(u)\|_{L^2(D)} \\
% &\rightarrow 0
% \end{aligned}
% \end{displaymath}
% as $n \rightarrow \infty$ for $* =c,u$. Since $x_n^s = (1-P_n^c -P_n^u) x_n$ 
% and $x^s = (1-P^c -P^u) x$, 
% we obtain the convergence $i_n(x_n^s) \rightarrow i(x^s)$.
% \end{proof}
%
%
%
%
\section{Convergence of unstable invariant manifolds}
\label{sec:convUnstaMan}
In this section, we prove upper and lower semicontinuity of local unstable invariant manifolds. 
We first show pointwise convergence of global unstable manifolds 
for the modified systems in Section \ref{subsec:convGlobUnstaModify}. 
Consequently, we prove Theorem \ref{th:mainResults} in Section \ref{subsec:UpLowSemicontUnsta}.
\subsection{Convergence of global unstable manifolds}
\label{subsec:convGlobUnstaModify}
%We first show convergence of global unstable manifolds for the modified systems.
%
Let 
\begin{displaymath}
 Y_n = \{ h \in C(X_n^{+}, X_n^{-}) : h(0) = 0 \text{ and } h \text{ is } \nu^{-1} \text{-Lipschitz } \}.
\end{displaymath}
Then $Y_n$ is a complete metric space with the norm
\begin{equation}
\label{eq:LipsNormYn}
  \|h\|_{\text{Lip}} = \sup_{w \ne 0} \frac{\|h(w)\|_{X_n^{-}}}{\|w\|_{X_n^{+}}}.
\end{equation}
We define $T_{t,n} : Y_n \rightarrow Y_n$  for $t \geq 0$ by
%\begin{displaymath}
  $T_{t,n}(h) = \tilde{h}$
%\end{displaymath}  
where $\tilde{h} \in Y_n$ such that  $\graph(\tilde{h})=\Phi_{t,n}(\graph(h))$.
Fix $t > 0$ sufficiently large such that
\begin{equation}
 \label{eq:lipsConstTtN}
  K :=  \nu (\nu -\mu)^{-1} \exp ((\alpha - \beta + \varepsilon (2+ \mu +\nu^{-1}))t) <1.
\end{equation} 
As in Theorem \ref{th:globalUnstaMan}, $T_{t,n}$ is a contraction on
$Y_n$ with a uniform contraction constant $K$ for all $n \in \mathbb N$. 
Moreover, $W_n^{+}$ is a graph of the fixed point $h_n^{+}$ of $T_{t,n}$.
To prove convergence of global unstable manifolds, we
show that the fixed point $h_n^{+}$ of $T_{t,n}$ converges to the fixed point
$h^{+}$ of $T_{t}$. 
\begin{lemma}
\label{lem:MoscoXminus}
 Suppose that Assumption \ref{assum:bddPertMosco} is satisfied. 
 For every  $v \in X^{-}$, there exists $v_n \in X_n^{-}$ such that
 $v_n \rightarrow v$ in $L^2(D)$. 
\end{lemma}
\begin{proof}
 Let $v \in X^{-} \subset L^2(\Omega)$. By the density of
 $H^1_0(\Omega)$ in $L^2(\Omega)$ and Mosco convergence assumption, it follows from a standard
 diagonal procedure that there exists $\xi_n \in
 H^1_0(\Omega_n)$ such that 
 %\begin{displaymath}
   $\xi_{n} \rightarrow v$
 %\end{displaymath} 
 in $L^2(D)$ as $n \rightarrow \infty$. By Lemma \ref{lem:convUsUcUu}, we get
$P_n^{-} \xi_n \rightarrow P^{-} v =  v$ in $L^2(D)$ as $n
\rightarrow \infty$. By taking $v_n := P_n^{-} \xi_n$, the lemma follows. 
\end{proof}
Let us define $h \in Y$ by
\begin{equation}
\label{eq:hForFixPt}
 h(w) := \frac{1}{C} h^{+}(w),
\end{equation} 
for all $w \in X^{+}$ where $C$ is a positive constant satisfying
\begin{equation}
\label{eq:constCHplus}
 \|P^{+}\|\|1-P_n^{+}\|M_1 M_2 \leq C,
\end{equation}
for all $n \in \mathbb N$.
Note that although $(1-P_n^{+})$ does not converge to $(1-P^{+})$
in $\mathscr L(L^2(D))$ under the operator norm, we use $\|1-P_n^{+}\| \leq 1 + \|P_n^{+}\|$
and \eqref{eq:convProplus} to obtain a bound $C$ above.  
%We will use that
%both $\|P_n^{+}\|$ and $\|P_n^{-}\| = \|1-P_n^{+}\|$ are bounded without
%further notice.
%
%
%
%

In the next lemma, we obtain an approximation of  $h$ by functions in $Y_n$.
\begin{lemma}
\label{lem:hNtoH}
 Let $h$ be as in \eqref{eq:hForFixPt}. There exists a sequence
 $\{h_n\}$ with $h_n \in Y_n$ for each $n \in \mathbb N$ such that
 \begin{itemize}
 \item[\upshape(i)] $h_n (P_n^{+} u|_{\Omega_n}) \rightarrow h (P^{+} u|_{\Omega})$ 
        in $L^2(D)$  as $n \rightarrow \infty$ for all $u \in L^2(D)$
 \item[\upshape(ii)] for each $m \in \mathbb N$, we have  $T_{t,n}^m(h_n) (P_n^{+} u|_{\Omega_n})
        \rightarrow T_{t}^m(h) (P^{+} u|_{\Omega})$ in $L^2(D)$ as $n \rightarrow \infty$
         for all $u  \in L^2(D)$.  
 \end{itemize}  
\end{lemma}
\begin{proof}
 We construct $h_n \in Y_n$ as follows.
 Define $h_n : X_n^{+}\rightarrow X_n^{-}$ by
\begin{equation}
 \label{eq:hNapprox}
   h_n(w) := \frac{1}{C}  (1-P_n^{+})\; \big ( h^{+} (P^{+} w|_{\Omega}) \big )|_{\Omega_n},
\end{equation}
for $w \in X_n^{+}$. It is clear that $h_n(0) = 0$. Moreover, for $w_1,
w_2 \in X_n^{+}$, it follows from the Lipschitz continuity of $h^{+}$ 
and the choice of $C$ in \eqref{eq:constCHplus} that 
\begin{displaymath}
\begin{aligned}
 &\|h_n(w_1) - h_n(w_2)\|_{X_n^{-}} \\
 %&\quad \leq M_1 \|h_n(w_1) - h_n(w_2)\|_{L^2(\Omega_n)} \\
 &\quad \leq M_1  \Big\| \frac{1}{C} (1-P_n^{+})\; \big ( h^{+} (P^{+} w_1|_{\Omega}) \big )|_{\Omega_n} 
      -\frac{1}{C}  (1-P_n^{+}) \; \big ( h^{+} (P^{+} w_2|_{\Omega}) \big )|_{\Omega_n} \Big \|_{L^2(\Omega_n)}
      \\
% &\quad \leq M_1 \frac{1}{C} \| 1-P_n^{+}\| \;
%       \|  h^{+} (P^{+} w_1|_{\Omega}) -  h^{+} (P^{+} w_2|_{\Omega})\|_{L^2(\Omega)} \\ 
 &\quad \leq  M_1 \frac{1}{C} \|1-P_n^{+}\ \| \; 
       \| h^{+} (P^{+} w_1|_{\Omega}) -  h^{+} (P^{+} w_2|_{\Omega})\|_{X^{-}}\\
 &\quad \leq M_1  \frac{1}{C} \nu^{-1}  \| 1-P_n^{+}\| \;
          \|P^{+} w_1|_{\Omega} -   P^{+} w_2|_{\Omega}\|_{X^{+}} \\ 
 %&\quad \leq M_1  \frac{1}{C}  \nu^{-1}  M_2  \| 1-P_n^{+}\| \;
 %        \|P^{+} w_1|_{\Omega} -   P^{+} w_2|_{\Omega}\|_{L^2(\Omega)} \\  
 &\quad \leq M_1 \frac{1}{C}  \nu^{-1} M_2 \|1-P_n^{+} \| \;  \|P^{+}\| \;
                \|w_1 -w_2\|_{L^2(\Omega_n)} \\
 &\quad \leq \nu^{-1}  \|w_1 -w_2\|_{X_n^{+}}.
\end{aligned}  
\end{displaymath}
Hence, $h_n$ is $\nu^{-1}-$Lipschitz and thus $h_n \in Y_n$. 
Note that we need to be careful about the norm used in the above calculation.
In particular, we take care of the equivalence of norms on $X^{-}$ and $X^{+}$ given in 
\eqref{eq:equivNormXminusplus}.
%and \eqref{eq:equivNormXplus}. 
This will be applied throughout the paper.  

We claim that $h_n$ defined above satisfies the properties (i) and (ii).
For (i), let $u \in L^2(D)$ be arbitrary.  By Lemma \ref{lem:MoscoXminus}, there exists 
$(v_n)_{n \in \mathbb N}$ with $v_n \in X_n^{-}$ such that 
\begin{equation}
 \label{eq:vNapproxHplus}
   v_n \rightarrow  h^{+} (P^{+} u|_{\Omega})
\end{equation}
in $L^2(D)$ as $n \rightarrow \infty$. 
We have from the triangle inequality that 
\begin{equation}
\label{eq:splitConvAppHnH}
 \begin{aligned}
  &\| h_n (P_n^{+} u|_{\Omega_n}) -  h (P^{+} u|_{\Omega}) \|_{L^2(D)}\\
   &\quad = \Big \| \frac{1}{C} (1-P_n^{+}) \;
          \Big (h^{+} \big (P^{+} (P_n^{+} u|_{\Omega_n})|_{\Omega} \big) \Big ) \Big|_{\Omega_n}
          - \frac{1}{C} h^{+} (P^{+} u|_{\Omega})   \Big \|_{L^2(D)} \\
   &\quad \leq \frac{1}{C}  
          \Big \| (1-P_n^{+}) \; \Big (h^{+} \big (P^{+} (P_n^{+} u|_{\Omega_n})|_{\Omega} \big) \Big ) \Big|_{\Omega_n}
            - (1-P_n^{+})\; (h^{+} (P^{+} u|_{\Omega}))|_{\Omega_n}  \Big \|_{L^2(D)} \\
   &\quad \quad + \frac{1}{C} \left \| (1-P_n^{+}) \; ( h^{+} (P^{+} u|_{\Omega}))|_{\Omega_n}
             -  h^{+} (P^{+} u|_{\Omega})   \right \|_{L^2(D)}. \\
 \end{aligned}
\end{equation}
%By the convergence of $P_n^{+}$ in \eqref{eq:convProplus}, we have 
%$\|1- P_n^{+}\| \leq 1 + \|P_n^{+}\| < \infty$ for all $n \in \mathbb N$.
%By equivalence of norms on $X^{-}$ and $X^{+}$ in \eqref{eq:equivNormXminus}, 
%\eqref{eq:equivNormXplus} and the convergence of $P_n^{+}$ in \eqref{eq:convProplus},  
%the first term on the right of \eqref{eq:splitConvAppHnH} satisfies
Using the equivalence of norms on $X^{-}$ and $X^{+}$, we can calculate
\begin{equation}
 \label{eq:firstTAppH}
  \begin{aligned}
  &\frac{1}{C} 
    \Big \| (1-P_n^{+}) \; \Big (h^{+} \big (P^{+} (P_n^{+} u|_{\Omega_n})|_{\Omega} \big) \Big ) \Big|_{\Omega_n}
     - (1-P_n^{+})\; (h^{+} (P^{+} u|_{\Omega}))|_{\Omega_n}  \Big \|_{L^2(D)} \\
  &\quad \leq \frac{1}{C}  \|1-P_n^{+}\|  \;  
    \|  h^{+} \big (P^{+} (P_n^{+} u|_{\Omega_n})|_{\Omega} \big) - h^{+} (P^{+} u|_{\Omega}) \|_{L^2(D)} \\
 % &\quad \leq \frac{1}{C} \|1-P_n^{+}\| \;
 %   \| h^{+} \big (P^{+} (P_n^{+} u|_{\Omega_n})|_{\Omega} \big) - h^{+} (P^{+} P^{+} u|_{\Omega}) \|_{L^2(\Omega)} \\
  &\quad \leq \frac{1}{C} \| 1-P_n^{+}\| \;
      \| h^{+} \big (P^{+} (P_n^{+} u|_{\Omega_n})|_{\Omega} \big) - h^{+} (P^{+} P^{+} u|_{\Omega})\|_{X^{-}} \\
  &\quad \leq \frac{1}{C} \nu^{-1} \| 1-P_n^{+} \| \;
      \| P^{+} (P_n^{+} u|_{\Omega_n})|_{\Omega} - P^{+} P^{+} u|_{\Omega} \|_{X^{+}} \\
  %&\quad \leq \frac{1}{C} \nu^{-1} M_2 \| 1-P_n^{+} \| \;
  %   \| P^{+} (P_n^{+} u|_{\Omega_n})|_{\Omega} - P^{+} P^{+} u|_{\Omega} \|_{L^2(\Omega)} \\ 
  &\quad \leq \frac{1}{C} \nu^{-1} M_2 \| 1-P_n^{+}\| \;
      \| P^{+} \| \; \| P_n^{+} u|_{\Omega_n} -  P^{+} u|_{\Omega} \|_{L^2(D)} \\
  &\quad \rightarrow 0
 \end{aligned}
\end{equation}
as $n \rightarrow \infty$, where we use \eqref{eq:convProplus} and
the boundedness of $\|1-P_n^{+}\|$ in the last step. 
For the second term on the right of \eqref{eq:splitConvAppHnH},
we use \eqref{eq:vNapproxHplus} and $(1-P_n^{+}) v_n = v_n$ to
obtain
\begin{equation}
 \label{eq:secondTAppH}
 \begin{aligned}
  &\frac{1}{C}  \left \| (1-P_n^{+}) \; ( h^{+} (P^{+} u|_{\Omega}))|_{\Omega_n}
             -  h^{+} (P^{+} u|_{\Omega})   \right \|_{L^2(D)} \\
  &\quad \leq \frac{1}{C} \| (1-P_n^{+}) \; ( h^{+} (P^{+} u|_{\Omega}))|_{\Omega_n} - v_n \|_{L^2(D)}\\
        &\quad \quad  + \frac{1}{C} \| v_n - h^{+} (P^{+} u|_{\Omega})   \|_{L^2(D)}\\
  %&\quad = \frac{1}{C} \| (1-P_n^{+}) \; ( h^{+} (P^{+} u|_{\Omega}))|_{\Omega_n} - 
  %         (1-P_n^{+}) v_n \|_{L^2(D)} \\
  %      &\quad \quad + \frac{1}{C} \| v_n - h^{+} (P^{+} u|_{\Omega})   \|_{L^2(D)}\\ 
  &\quad \leq \frac{1}{C} \| 1-P_n^{+}\| \; \| h^{+} (P^{+} u|_{\Omega}) -  v_n \|_{L^2(D)} \\
        &\quad \quad  + \frac{1}{C} \| v_n  - h^{+} (P^{+} u|_{\Omega})   \|_{L^2(D)}\\ 
  &\quad \rightarrow 0
 \end{aligned}
\end{equation}
as $n \rightarrow \infty$. It follows from \eqref{eq:splitConvAppHnH} -- \eqref{eq:secondTAppH} that 
\begin{displaymath}
 \| h_n (P_n^{+} u|_{\Omega_n}) -  h (P^{+} u|_{\Omega}) \|_{L^2(D)} \rightarrow 0
\end{displaymath} 
as $n \rightarrow \infty$. Since the above argument is valid for any $u \in L^2(D)$, 
statement (i) follows.

We next prove (ii) by induction on $m \in \mathbb N$. By part (i) of this
proof, the property (ii) is true for $m=0$. For induction step, assume that
\begin{displaymath}
 T_{t,n}^m(h_n) (P_n^{+} u|_{\Omega_n}) \rightarrow T_{t}^m(h) (P^{+} u|_{\Omega})
\end{displaymath}
in $L^2(D)$ as $n \rightarrow \infty$ for all $u \in L^2(D)$ holds true for $m =
0,1, \ldots, k$. We need to show that 
\begin{equation}
\label{eq:inductionFlow}
 T_{t,n}^{k+1}(h_n) (P_n^{+} u|_{\Omega_n}) \rightarrow T_{t}^{k+1}(h) (P^{+} u|_{\Omega})
\end{equation}
in $L^2(D)$ as $n \rightarrow \infty$ for all $u \in L^2(D)$.
Let $u \in L^2(D)$ be arbitrary. We set $w:=  P^{+} u|_{\Omega} \in X^{+}$ and 
$w_n:=  P_n^{+} u|_{\Omega_n} \in X_n^{+}$. It follows from \eqref{eq:convProplus} that
\begin{equation}
\label{eq:wNconv}
   w_n \rightarrow w
\end{equation}
in $L^2(D)$ as $n \rightarrow \infty$.
Since $\text{graph}(T_t^{k+1} (h)) =
\Phi_t(\text{graph}(T_t^{k} (h))$, there exists $w_0 \in X^{+}$ such
that
\begin{displaymath}
 \Phi_t \left( w_0 \oplus T_t^k (h)(w_0) \right ) = w \oplus  T_t^{k+1} (h)(w).
\end{displaymath}
For each $n \in \mathbb N$, we define $w_{0,n} := P_n^{+} w_0|_{\Omega_n}$. 
%By the convergence of spectral projections $P_n^{+}$ in \eqref{eq:convProplus}
Again, by \eqref{eq:convProplus}, we have  $w_{0,n} \rightarrow w_0$
in $L^2(D)$ as $n \rightarrow \infty$. Moreover, by the induction hypothesis, 
\begin{displaymath}
 T_{t,n}^k(h_n) (w_{0,n}) = T_{t,n}^k(h_n) (P_n^{+} w_0|_{\Omega_n}) 
                            \rightarrow T_{t}^k(h) (P^{+} w_0)
                            =  T_{t}^k(h) (w_0)
\end{displaymath}
in $L^2(D)$ as $n \rightarrow \infty$. Hence, it follows from 
%the convergence of solutions under domain perturbation in 
\eqref{eq:convSolUsingSemiflow} that
\begin{displaymath}
 \Phi_{t,n} \left( w_{0,n} \oplus T_{t,n}^k(h_n) (w_{0,n}) \right )
 \rightarrow \Phi_{t} \left( w_{0} \oplus T_{t}^k(h) (w_{0}) \right )
 = \left (w \oplus  T_t^{k+1} (h)(w) \right )
\end{displaymath}
in $L^2(D)$ as $n \rightarrow \infty$. Since $\text{graph}(T_{t,n}^{k+1} (h_n)) =
\Phi_{t,n}(\text{graph}(T_{t,n}^{k} (h_n))$, there exists $\xi_n \in X_n^{+}$ 
such that 
\begin{displaymath}
 \Phi_{t,n} \left( w_{0,n} \oplus T_{t,n}^k (h_n)(w_{0,n}) \right ) = \xi_n \oplus  T_{t,n}^{k+1} (h_n)(\xi_n),
\end{displaymath}
for each $n \in \mathbb N$. Hence, 
\begin{equation}
\label{eq:flowXInW}
     \xi_n \oplus  T_{t,n}^{k+1} (h_n)(\xi_n) 
 \rightarrow   w \oplus  T_t^{k+1} (h)(w) 
\end{equation}
in $L^2(D)$ as $n \rightarrow \infty$.
%
%By convergence of the spectral projections \eqref{eq:convProplus} and \eqref{eq:convProminus} ,
%
By Lemma \ref{lem:convUsUcUu}, it follows from \eqref{eq:flowXInW}
that
\begin{equation}
\label{eq:PplusXInW}
    \xi_n \rightarrow w
\end{equation}
and 
\begin{equation}
\label{eq:PminusXInW}
   T_{t,n}^{k+1} (h_n)(\xi_n) 
 \rightarrow    T_t^{k+1} (h)(w) 
\end{equation}
in $L^2(D)$ as $n \rightarrow \infty$. We obtain from \eqref{eq:wNconv} and
\eqref{eq:PplusXInW} that
%\begin{displaymath}
  $\| \xi_n - w_n \|_{L^2(D)} \rightarrow 0$
%\end{displaymath}
as $n \rightarrow \infty$. Since $T_{t,n}^{k+1}(h_n)$ is $\nu^{-1}$-Lipschitz,
it follows that
\begin{equation}
\label{eq:TmapXInW}
\begin{aligned}
  \left \| T_{t,n}^{k+1}(h_n) (\xi_n) - T_{t,n}^{k+1}(h_n) (w_n) \right \|_{L^2(\Omega_n)}
  & \leq \left \| T_{t,n}^{k+1}(h_n) (\xi_n) - T_{t,n}^{k+1}(h_n) (w_n) \right \|_{X_n^{-}} \\
  & \leq \nu^{-1} \| \xi_n - w_n  \|_{X_n^{+}} \\
  & \leq \nu^{-1} M_2 \| \xi_n - w_n  \|_{L^2(\Omega_n)} \\
  & \rightarrow 0 
\end{aligned}
\end{equation}
as $n \rightarrow \infty$. By definitions of $w_n$ and $w$ together with \eqref{eq:PminusXInW} and 
\eqref{eq:TmapXInW}, we conclude that
\begin{displaymath}
 \begin{aligned}
  &\left \| T_{t,n}^{k+1}(h_n) (P_n^{+} u|_{\Omega_n}) -  T_{t}^{k+1}(h) (P^{+} u|_{\Omega}) \right \|_{L^2(D)} \\
  & \quad = \left \| T_{t,n}^{k+1}(h_n) (w_n) - T_{t}^{k+1}(h) (w) \right \|_{L^2(D)} \\
  & \quad \leq \left \| T_{t,n}^{k+1}(h_n) (w_n) - T_{t,n}^{k+1}(h_n) (\xi_n) \right \|_{L^2(D)} \\
  & \quad \quad  + \left \| T_{t,n}^{k+1}(h_n) (\xi_n)  - T_{t}^{k+1}(h) (w) \right \|_{L^2(D)} \\
  & \quad \rightarrow 0
 \end{aligned}
\end{displaymath}
as $n \rightarrow \infty$. As $u \in L^2(D)$ was arbitrary, we have shown \eqref{eq:inductionFlow}.
\end{proof}
%  
%
%
% We next recall a standard result on the rate of convergence to 
% the fixed point of a contraction mapping below (see for example \cite[Remark 1.2.3 (ii)]{MR0467717}).
% \begin{lemma}
% \label{lem:fixedPtConvRate}
% Let $(X,d)$ be a complete metric space and $T:X \rightarrow X$ be a contraction
% mapping with a contraction constant $k \in [0,1)$, that is,
% \begin{displaymath} 
%  d(Tx_1,Tx_2) \leq k d(x_1,x_2),
% \end{displaymath}
% for all $x_1,x_2 \in X$. Denote by $x^{*}$ the unique fixed point of $T$. Then for 
% any $x \in X$,
% \begin{equation} 
%  \label{eq:fixedPtRate}
%   d( x^{*},T^n x) \leq \frac{k^n}{1-k} d(x,Tx),
% \end{equation}
% for all $n \in \mathbb N$.
% \end{lemma}
%
%
%
%
We prove the pointwise convergence of global unstable invariant manifolds in the
following theorem.  
\begin{theorem}
\label{th:ptwiseConvGlobUnstMan}
Assume that all assumptions in Theorem \ref{th:mainResults} are satisfied
and $H^1_0(\Omega_n)$ converges to $H^1_0(\Omega)$ in the sense of Mosco.
Then we have
\begin{displaymath}
   h_n^{+} (P_n^{+} u|_{\Omega_n}) \rightarrow  h^{+} (P^{+} u|_{\Omega})
\end{displaymath}
in $L^2(D)$ as $n \rightarrow \infty$ for all $u \in L^2(D)$.
\end{theorem}
\begin{proof}
 Fix $u \in L^2(D)$ and let $\zeta >0$ be arbitrary. We can choose 
 $m_0 \in \mathbb N$ independent of $n$ such that the contraction
 constant $K$ in \eqref{eq:lipsConstTtN} satisfies
\begin{equation}
 \label{eq:m0ForHnPlus}
 \max \left \{ \sup_{n \in \mathbb N} \left \{ \frac{K^{m_0}}{1-K} 2 \nu^{-1} \|P_n^{+} u|_{\Omega_n}\|_{X_n^{+}} \right \},
   \frac{K^{m_0}}{1-K} 2 \nu^{-1} \|P^{+} u|_{\Omega}\|_{X^{+}} \right \}
    \leq \frac{\zeta}{3}. 
\end{equation} 
%for all $n \in \mathbb N$ and 
% \begin{equation}
%  \label{eq:m0ForHPlus}
%   \frac{K^{m_0}}{1-K} 2 \nu^{-1} \|P^{+} u|_{\Omega}\|_{X^{+}} \leq \frac{\zeta}{3}. 
% \end{equation} 
We take $h_n \in Y_n$ and $h \in Y$ as in Lemma \ref{lem:hNtoH}. Then
by the definition of Lip-norm on $Y$ and $Y_n$ (see \eqref{eq:LipsNormY} and 
\eqref{eq:LipsNormYn}, respectively), we see that
\begin{equation}
 \label{eq:splitHnPlusHPlus}
 \begin{aligned}
 &\left \| h_n^{+} (P_n^{+} u|_{\Omega_n}) -  h^{+} (P^{+} u|_{\Omega}) \right \|_{L^2(D)} \\ 
%  &\quad \leq \left \| h_n^{+} (P_n^{+} u|_{\Omega_n}) - T_{t,n}^{m_0} (h_n) (P_n^{+} u|_{\Omega_n}) \right \|_{L^2(D)} \\
%  &\quad \quad +\left \| T_{t,n}^{m_0} (h_n) (P_n^{+} u|_{\Omega_n}) - T_{t}^{m_0} (h) (P^{+} u|_{\Omega}) \right \|_{L^2(D)} \\
%  &\quad \quad +\left \| T_{t}^{m_0} (h) (P^{+} u|_{\Omega})- h^{+} (P^{+} u|_{\Omega}) \right \|_{L^2(D)} \\ 
 &\quad \leq \left \| h_n^{+} (P_n^{+} u|_{\Omega_n}) - T_{t,n}^{m_0} (h_n) (P_n^{+} u|_{\Omega_n}) \right \|_{X_n^{-}} \\
 &\quad \quad +\left \| T_{t,n}^{m_0} (h_n) (P_n^{+} u|_{\Omega_n}) - T_{t}^{m_0} (h) (P^{+} u|_{\Omega}) \right \|_{L^2(D)} \\
 &\quad \quad +\left \| T_{t}^{m_0} (h) (P^{+} u|_{\Omega})- h^{+} (P^{+} u|_{\Omega}) \right \|_{X^{-}} \\ 
 &\quad \leq  \| h_n^{+} -  T_{t,n}^{m_0} (h_n)\|_{\text{Lip}} \| P_n^{+} u|_{\Omega_n} \|_{X_n^{+}} \\
 &\quad \quad +\left \| T_{t,n}^{m_0} (h_n) (P_n^{+} u|_{\Omega_n}) - T_{t}^{m_0} (h) (P^{+} u|_{\Omega}) \right \|_{L^2(D)} \\
 &\quad \quad +\| T_{t}^{m_0} (h) -  h^{+} \|_{\text{Lip}} \| P^{+} u|_{\Omega} \|_{X^{+}},
 \end{aligned}
\end{equation}
for all $n \in \mathbb N$.
By an elementary result on the rate of convergence to the fixed point of
a contraction mapping (see e.g. \cite[Remark 1.2.3 (ii)]{MR0467717}), we have
\begin{equation}
 \label{eq:rateTm0H}
  \| h^{+} - T_{t}^{m_0}(h) \|_{\text{Lip}} 
  \leq \frac{K^{m_0}}{1-K} \| h - T_{t}(h)\|_{\text{Lip}}
  \leq \frac{K^{m_0}}{1-K} 2 \nu^{-1}
\end{equation} 
and
\begin{equation}
 \label{eq:rateTm0Hn}
  \| h_n^{+} - T_{t,n}^{m_0}(h_n) \|_{\text{Lip}} 
  \leq \frac{K^{m_0}}{1-K} \| h_n - T_{t,n}(h_n)\|_{\text{Lip}}
  \leq \frac{K^{m_0}}{1-K} 2 \nu^{-1},
\end{equation} 
for all $n \in \mathbb N$. Moreover, Lemma \ref{lem:hNtoH} (ii) implies that
there exists $N_0 \in \mathbb N$ such that 
\begin{equation}
 \label{eq:HnHepsilon}
 \left \| T_{t,n}^{m_0} (h_n) (P_n^{+} u|_{\Omega_n}) -  T_{t}^{m_0} (h) (P^{+} u|_{\Omega}) \right \|_{L^2(D)} 
 \leq \frac{\zeta}{3},
\end{equation}
for all $n > N_0$. It follows from \eqref{eq:splitHnPlusHPlus} -- \eqref{eq:HnHepsilon} that
\begin{displaymath}
\begin{aligned}
 \left \|  h_n^{+} (P_n^{+} u|_{\Omega_n}) -  h^{+} (P^{+} u|_{\Omega}) \right \|_{L^2(D)}
 &\leq \frac{K^{m_0}}{1-K} 2 \nu^{-1} \| P_n^{+} u|_{\Omega_n} \|_{X_n^{+}}
   + \frac{\zeta}{3} \\
 &\quad  + \frac{K^{m_0}}{1-K} 2 \nu^{-1} \| P^{+} u|_{\Omega} \|_{X^{+}},
\end{aligned}
\end{displaymath}
for all $n > N_0$. By our choice of $m_0$ in \eqref{eq:m0ForHnPlus},
%and \eqref{eq:m0ForHPlus},  
we conclude that 
\begin{displaymath}
 \left \| h_n^{+} (P_n^{+} u|_{\Omega_n}) -  h^{+} (P^{+} u|_{\Omega}) \right \|_{L^2(D)}
 \leq \zeta, 
\end{displaymath}
for all $n > N_0$. As $\zeta > 0$ was arbitrary, we get
%\begin{displaymath}
  $h_n^{+} (P_n^{+} u|_{\Omega_n}) \rightarrow  h^{+} (P^{+} u|_{\Omega})$
%\end{displaymath}
in $L^2(D)$ as $n \rightarrow \infty$. Since this argument works for any
$u \in L^2(D)$, the statement of the theorem follows.
\end{proof}
\subsection{Upper and lower semicontinuity of local unstable manifolds}
\label{subsec:UpLowSemicontUnsta}
We are now in the position to prove Theorem \ref{th:mainResults}.
\begin{proof}[Proof of Theorem \ref{th:mainResults} (ii)]
 As discussed at the end of Section \ref{subsec:existInvPert}, there exist $\delta_1$
 and $\delta_2$ such that 
 $W_n^u = W_n^{u}(U_n)$ is a local unstable invariant manifold where $U_n = V_{1,n} \times V_{2,n}$ 
 with $V_{1,n}$ is a ball of radius $\delta_1$ in $X_n^{-}$ 
 and $V_{2,n}$ is a ball of radius $\delta_2$ in $X_n^{+}$ for all $n
 \in \mathbb N$. Moreover, a similar statement holds for the unperturbed
 problem. By the equivalence of norms on $X_n^{-}$ and $X_n^{+}$ with
 uniform parameters $\alpha$ and $\beta$, we can chose $\delta > 0$ such that
 $B_n := B_{L^2(\Omega_n)}(0, \delta) \subset V_{1,n} \times V_{2,n}$
 for all $n \in \mathbb N$ and $B := B_{L^2(\Omega)}(0, \delta) \subset
 V_{1} \times V_{2}$. 
 
 To prove the lower semicontinuity, we show that for every $\zeta > 0$, there exists $N_0 \in \mathbb N$
 independent of $u \in \text{\upshape{graph}}(h^{+})\cap B$ such
 that
 \begin{displaymath}
   \inf_{v \in \text{\upshape{graph}}(h_n^{+})\cap B_n}  \| u - v \|_{L^2(D)}
   < \zeta,
 \end{displaymath}
 for all $n > N_0$ and for all $u \in \text{\upshape{graph}}(h^{+})\cap B$.
  Let $\zeta >0$ be arbitrary. By the Lipschitz continuity of $h^{+}:X^{+}
 \rightarrow X^{-}$ (taking \eqref{eq:equivNormXminusplus} into account),
 %%and \eqref{eq:equivNormXplus} into account), 
 we have that
 for every $w_0 \in X^{+}$, there exists  $\rho > 0$ such that 
 \begin{equation}
 \label{eq:LipsContHplusApp}
   \|(w \oplus h^{+}(w)) - (w_0 \oplus h^{+}(w_0))\|_{L^2(\Omega)} < \frac{\zeta}{2},
 \end{equation}
for all $w \in B_{X^{+}}(w_0,\rho) := \{w \in X^{+} :
\|w-w_0\|_{L^2(\Omega)} < \rho \}$. Note that $\rho$ is independent of
$w_0 \in X^{+}$. 
We set
\begin{displaymath}
  W := P^{+} \left (\text{graph}(h^{+}) \cap
       B \right)
    = \{ w \in X^{+}: w \oplus h^{+}(w) \in  B \}.
\end{displaymath}
Since $\dim(X^{+}) < \infty$, the set $\overline{W}$ is compact. 
% If we take the open cover $\{ B_{X^{+}}(w,\rho) : w \in W\}$ of
% $\overline{W}$, then there exists a finite subcover
Hence, we can choose a finite cover
%\begin{displaymath}
 $\{ B_{X^{+}}(w_k,\rho) : w_k \in W , k = 1,\ldots, m\}$
%\end{displaymath}
%Hence,
of $\overline W$ so that  
\begin{equation}
\label{eq:subCoverPplus}
  W \subset \bigcup_{k=1}^{m} B_{X^{+}}(w_k,\rho).
\end{equation}
Denoted by 
$\Delta := \min\{\delta - \|w_k \oplus h^{+}(w_k)\|_{L^2(\Omega)}: k=1,\ldots, m \}$.
Setting $w_{k,n} := P_n^{+}  w_k|_{\Omega_n} \in X_n^{+}$ for $n \in \mathbb N$ and 
$k = 1, \ldots, m$. %It follows from convergence of spectral projections in \eqref{eq:convProplus}
We have from \eqref{eq:convProplus} that $w_{k,n} \rightarrow w_k$ in $L^2(D)$ as $n \rightarrow \infty$
for each $k =1, \ldots,m$. Moreover, by Theorem \ref{th:ptwiseConvGlobUnstMan} 
$h_n^{+}(w_{k,n}) \rightarrow   h^{+}(w_k)$ in $L^2(D)$ as 
$n \rightarrow \infty$ for each $k =1,\ldots,m$.
% Hence, for each $k =1, \ldots, m$ we can find $N_k \in \mathbb N$ such that
% \begin{displaymath}
%  \| (w_{k,n} \oplus h_n^{+}(w_{k,n})) -  (w_k \oplus h^{+}(w_k)) \|_{L^2(D)} 
%  < \min\left \{\frac{\zeta}{2}, \Delta \right \},
% \end{displaymath}
% for all $n >  N_k$. Setting $N_0 = \max\{N_k : k = 1, \ldots, m \}$, we get
Hence, we can find $N_0 \in \mathbb N$ such that
\begin{equation}
 \label{eq:estHNplusHplusZetaWk}
 \| (w_{k,n} \oplus h_n^{+}(w_{k,n})) -  (w_k \oplus h^{+}(w_k)) \|_{L^2(D)} 
 < \min\left \{\frac{\zeta}{2}, \Delta \right \},
\end{equation}
for all $n > N_0$ and for all $ k = 1, \ldots, m$.
Using \eqref{eq:estHNplusHplusZetaWk}, we have
\begin{displaymath}
\begin{aligned}
 \|w_{k,n} \oplus h_n^{+}(w_{k,n}) \|_{L^2(\Omega_n)}
 &\leq  \|(w_{k,n} \oplus h_n^{+}(w_{k,n})) -  (w_k \oplus h^{+}(w_k))\|_{L^2(D)}\\
 &\quad  + \| w_k \oplus h_n^{+}(w_k) \|_{L^2(\Omega)} \\
 &<  \| w_k \oplus h_n^{+}(w_k) \|_{L^2(\Omega)} + \Delta \\
 &\leq \| w_k \oplus h_n^{+}(w_k) \|_{L^2(\Omega)} +(\delta - \|w_k \oplus h_n^{+}(w_k) \|_{L^2(\Omega)})\\
 &= \delta,
\end{aligned} 
\end{displaymath}
for all $n > N_0$ and for all $ k = 1, \ldots, m$. Hence,
$w_{k,n} \oplus h_n^{+}(w_{k,n}) \in \text{\upshape{graph}}(h_n^{+})\cap B_n$
for all $n > N_0$ and for all $ k = 1, \ldots, m$.
Let $u$ be in $\text{\upshape{graph}}(h^{+})\cap B$ and write $u = w \oplus h^{+}(w)$
for some $w \in W$. By \eqref{eq:subCoverPplus}, there exists $k \in \{1,\ldots,m\}$
such that $w \in B_{X^{+}}(w_k, \rho)$. It follows from \eqref{eq:LipsContHplusApp}
and \eqref{eq:estHNplusHplusZetaWk} that
\begin{displaymath}
\begin{aligned}
  &\| (w_{k,n} \oplus h_n^{+}(w_{k,n})) -  (w \oplus h^{+}(w)) \|_{L^2(D)} \\
  &\quad \leq  \| (w_{k,n} \oplus h_n^{+}(w_{k,n})) -  (w_k \oplus h^{+}(w_k)) \|_{L^2(D)} \\
  & \quad \quad + \| (w_k \oplus h^{+}(w_k)) - (w \oplus h^{+}(w)) \|_{L^2(D)} \\
  &\quad < \frac{\zeta}{2} + \frac{\zeta}{2} \\
  &\quad = \zeta,  
\end{aligned}
\end{displaymath}
for all $n > N_0$.  Since $w_{k,n} \oplus h_n^{+}(w_{k,n}) \in \text{\upshape{graph}}(h_n^{+})\cap B_n$
for all $n > N_0$, we get
\begin{displaymath}
  \inf_{v \in \text{\upshape{graph}}(h_n^{+})\cap B_n}  \| u - v  \|_{L^2(D)} < \zeta,
\end{displaymath}
for all $n > N_0$. The above estimate holds for every 
$u = w \oplus h^{+}(w) \in \text{graph}(h^{+})\cap B$ and notice that $N_0$ is 
independent of $u$.  
% Hence, 
% \begin{displaymath}
%   \inf_{v \in \text{\upshape{graph}}(h_n^{+})\cap B_n}  \| u - v \|_{L^2(D)} < \zeta,
% \end{displaymath}
% for all $n > N_0$ and for all $u \in \text{graph}(h^{+})\cap B$. 
As $\zeta > 0$ was  arbitrary, we obtain the lower semicontinuity.
\end{proof}
Using our characterisation in Lemma \ref{lem:upperSemiContCharac}, we can show the upper 
semicontinuity of unstable invariant manifolds.
\begin{proof}[Proof of Theorem \ref{th:mainResults} (i)]
 We consider the same neighbourhood $B_n$ and $B$ as in the proof above.
%  By Lemma \ref{lem:upperSemiContCharac}, we need to show that 
%  for any sequence $\{\xi_n\}_{n \in \mathbb N}$ with 
%  $\xi_n \in \text{\upshape{graph}}(h_n^{+})\cap B_n$, if $\{\xi_{n_k} \}_{k \in \mathbb N}$ is
%  a subsequence, then there exist a further subsequence (denoted again by $\xi_{n_k}$)
%  and  a sequence $\{u_{n_k}\}_{k \in \mathbb N}$ with $u_{n_k} \in \text{\upshape{graph}}(h^{+})\cap B$ 
%  such that $\| \xi_{n_k} - u_{n_k}\|_{L^2(D)} \rightarrow 0$  as $k \rightarrow
%  \infty$.  
%
 Let $\{\xi_n\}_{n \in \mathbb N}$ be a sequence with $\xi_n \in
 \text{\upshape{graph}}(h_n^{+})\cap B_n$ and $(\xi_{n_k})_{k \in \mathbb N}$ 
 be an arbitrary subsequence. We write $\xi_{n_k} := w_{n_k} \oplus
 h_{n_k}^{+}(w_{n_k})$ for some $w_{n_k} \in X_{n_k}^{+}$. Since
 $\|\xi_{n_k}\|_{L^2(\Omega_{n_k})} = \| w_{n_k} \oplus h_{n_k}^{+}(w_{n_k}) \|_{L^2(\Omega_{n_k})} < \delta$
 for all $k \in \mathbb N$, we can apply Corollary \ref{cor:convBoundSeqInXplus} to extract a subsequence
 of $\{w_{n_k}\}_{k \in \mathbb N}$ (indexed again by $n_k$) such that
 %\begin{displaymath}
      $w_{n_k} \rightarrow  w$ 
 %\end{displaymath}
 in $L^2(D)$ with the limit $w \in X^{+}$. Hence, by the Lipschitz continuity
 of $h_n^{+}$ and Theorem \ref{th:ptwiseConvGlobUnstMan}, we get
 \begin{displaymath}
 \begin{aligned}
  \| h_{n_k}^{+} (w_{n_k}) - h^{+}(w)\|_{L^2(D)} 
  &\leq  \| h_{n_k}^{+} (w_{n_k}) -   h_{n_k}^{+} (P_{n_k}^{+} w|_{\Omega_{n_k}}) \|_{L^2(D)}\\
  & \quad +\| h_{n_k}^{+} (P_{n_k}^{+}  w|_{\Omega_{n_k}})  -   h^{+}(w)\|_{L^2(D)}\\
  & \rightarrow 0
 \end{aligned}
\end{displaymath}
as $k \rightarrow \infty$. If we set $u := w \oplus h^{+}(w) \in \text{graph}(h^{+})$, then
$\xi_{n_k} \rightarrow u$ in $L^2(D)$ as $k \rightarrow
\infty$. Since $\|\xi_{n_k}\|_{L^2(D)} < \delta$ for all $k \in
\mathbb N$, we get $\|u\|_{L^2(D)} \leq \delta$. Hence, $u \in
\text{graph}(h^{+}) \cap \overline{B} = \overline{\text{graph}(h^{+})\cap B}$.
We can find $u_{n_k} \in \text{graph}(h^{+}) \cap B$ such
that $u_{n_k} \rightarrow u$ in $L^2(\Omega)$ as $k \rightarrow \infty$.
Therefore,
\begin{displaymath}
 \| \xi_{n_k} - u_{n_k}\|_{L^2(D)} 
 \leq  \| \xi_{n_k} - u\|_{L^2(D)} + \| u - u_{n_k}\|_{L^2(D)}
  \rightarrow 0
\end{displaymath}
as $k \rightarrow \infty$.
By Lemma \ref{lem:upperSemiContCharac}, the statement in Theorem \ref{th:mainResults} (i)
follows.
\end{proof} 
\section{Convergence of stable invariant manifolds}
\label{sec:convStaMan}

Recall that the local stable manifold is a graph of Lipschitz function
$h^{-}: X^{-} \rightarrow X^{+}$ inside a suitable product neightbourhood
of $0 \in L^2(\Omega)$ determined by
the modification in the construction (Theorem \ref{th:locStaMan}). 
In this section, we prove the upper and lower semicontinuity of local stable invariant
manifolds with the following modification. 

 Fix the renorming of $X_n^{-}$, $X_n^{+}$, $X^{-}$ and $X^{+}$ 
 (see \eqref{eq:normXminusplus}) using
 the same parameters $\alpha$ and $\beta$ for all $n \in \mathbb N$. 
 %(see Bates and Jones [Lemma 2.1]).
 By shrinking the neighbourhood (choosing a smaller Lipschitz constant
 $\varepsilon$ for the nonlinear terms $f_n$ and $f$), we can make the 
 following assumption.

\begin{assumption}
 We assume that 
 \begin{equation}
 \label{eq:mu0}
   0 <  \mu_0 < \inf \left \{ \frac{1}{2(\|P^{+}\|+ \|P^{-}\|)},
  \frac{1}{2(\|P_n^{+}\|+ \|P_n^{-}\|)} : n \in \mathbb N \right \}
 \end{equation}
 and
 \begin{equation}
 \label{eq:mu}
  \mu := \frac{\mu_0}{M_1 M_2}
 \end{equation}
  are parameters such that both $\mu_0$ and $\mu$ 
 satisfy the conditions for $\mu$ in \eqref{eq:constMuNu1} and \eqref{eq:constMuNu2}. %[Bates and Jones (2.13) and (2.14)]. 
\end{assumption} 

We denote the Lipschitz functions for the modification $\mu_0$ by $\hat{h}^{-}$ and for
the modification $\mu$ by $h^{-}$. Let $U$ be a smaller product neighbourhood of $0$
in $L^2(\Omega)$ such that both modifications agree. Hence, the local stable manifold is  
$W^{s}(U) := \text{graph}(h^{-}) \cap U = \text{graph}(\hat{h}^{-}) \cap U$.
Similarly, for each $n \in \mathbb N$, 
we denote the Lipschitz functions for the modification $\mu_0$ by $\hat{h}_n^{-}$ and for
the modification $\mu$ by $h_n^{-}$.  As discussed at the end of Section \ref{subsec:existInvPert}, 
we can take a uniform product neighbourhood $U_n$ of $0$ in $L^2(\Omega_n)$ 
such that both modifications agree.  Hence, the local stable manifold is  
$W_n^{s}(U_n) := \text{graph}(h_n^{-}) \cap U_n = \text{graph}(\hat{h}_n^{-}) \cap U_n$.
We choose $\delta > 0$  so that $\overline{B} \subset U$ and $\overline{B}_n \subset U_n$, 
where $B:=B_{L^2(\Omega)}(0,\delta)$ and $B_n:=B_{L^2(\Omega_n)}(0,\delta)$. 
Hence, $h^{-} (v) = \hat{h}^{-}(v)$ on $\overline{B}$ 
and $h_n^{-} (v) = \hat{h}_n^{-}(v)$ on $\overline{B}_n$. 
We prove Theorem \ref{th:mainResults2} by taking the balls of radius $\delta$ chosen above.

\begin{lemma}
\label{lem:seqOfGphSmallerBall}
Let $\delta > 0$ and $\zeta_n > 0$ be a sequence with $\zeta_n \rightarrow 0$ 
as $n \rightarrow \infty$. We write $B:=B_{L^2(\Omega)}(0,\delta)$ and $B_n:=B_{L^2(\Omega_n)}(0,\delta)$. 
\begin{itemize}
\item[\upshape{(i)}] If $z_n = y_n \oplus h^{-}(y_n)$ is a sequence 
in $\text{\upshape{graph}}(h^{-})$ with $z_n \in B_{L^2(\Omega)}(0, \delta + \zeta_n)$
for each $n \in \mathbb N$, then there exist a subsequence $z_{n_k}$ and a sequence
$u_{n_k}$ in $\graph(h^{-}) \cap B$
%$\text{\upshape{graph}}(h^{-}) \cap B_{L^2(\Omega)}(0, \delta)$ 
such that
%\begin{displaymath}
 $\| z_{n_k} - u_{n_k} \|_{L^2(\Omega)} \rightarrow 0$ 
%\end{displaymath} 
as $k \rightarrow \infty$.
\item[\upshape{(ii)}] If $z_n = y_n \oplus h^{-}(y_n)$ is a sequence 
with $z_n \in \text{\upshape{graph}}(h_n^{-}) \cap B_{L^2(\Omega_n)}(0, \delta + \zeta_n)$
for each $n \in \mathbb N$, then there exist a subsequence $z_{n_k}$ and a sequence
$u_{n_k}$ with $u_{n_k} \in \graph(h_{n_k}^{-}) \cap B_{n_k}$ for each $k \in \mathbb N$
%$\text{\upshape{graph}}(h_{n_k}^{-}) \cap B_{L^2(\Omega_{n_k})}(0, \delta)$ 
such that
%\begin{displaymath}
 $\| z_{n_k} - u_{n_k} \|_{L^2(\Omega_{n_k})} \rightarrow 0$ 
%\end{displaymath} 
as $k \rightarrow \infty$.
\end{itemize}
\end{lemma}
\begin{proof}
For assertion (i), using \eqref{eq:mu0} we can fix $b >0$ such that
\begin{equation}
 \label{eq:constB} 
  b > \frac{1}{(\|P^{+}\| + \|P^{-}\|)^{-1} -2 \mu_0}.
\end{equation}
Since $\zeta_n \rightarrow 0$, we can find $N_0 \in \mathbb N$ such that 
$\zeta_n < \delta / b$  for all $n > N_0$. 
We extract a subsequence $\zeta_{n_k}$ so that $\zeta_{n_k} < \delta /b$
for all $k \in \mathbb N$.
Define 
\begin{equation}
\label{eq:constAnk}
 a_{n_k} := 1 - \frac{b \zeta_{n_k}}{\|y_{n_k}\|_{L^2(\Omega)} + \|h^{-}(y_{n_k})\|_{L^2(\Omega)}}, 
\end{equation}
for each $k \in \mathbb N$. By our assumptions, 
$\| z_{n_k}\|_{L^2(\Omega)} =\| y_{n_k} \oplus h^{-}(y_{n_k})\|_{L^2(\Omega)} < \delta + \zeta_{n_k}$
for all $k \in \mathbb N$. If $\| z_{n_k}\|_{L^2(\Omega)} \geq \delta$, then
\begin{displaymath}
 \| y_{n_k}\|_{L^2(\Omega)} + \| h^{-}(y_{n_k})\|_{L^2(\Omega)} 
 \geq \| y_{n_k} \oplus h^{-}(y_{n_k})\|_{L^2(\Omega)} 
 \geq \delta.
\end{displaymath}
Since $\zeta_{n_k} < \delta/b$, if $\| z_{n_k}\|_{L^2(\Omega)} \geq \delta$ we have that
\begin{displaymath}
\frac{b \zeta_{n_k}}{\|y_{n_k}\|_{L^2(\Omega)} + \|h^{-}(y_{n_k})\|_{L^2(\Omega)}} 
< \frac{b (\delta/b)}{\delta}
= 1.
\end{displaymath}
It follows from \eqref{eq:constAnk} that $0 < a_{n_k} \leq 1$ if $\| z_{n_k}\|_{L^2(\Omega)} \geq \delta$.
For each $k \in \mathbb N$, we define $u_{n_k} \in \text{graph}(h^{-})$ by
\begin{equation}
\label{eq:seqUnkInB}
 u_{n_k} := \left \{ 
            \begin{aligned}
             &z_{n_k}  &&\quad \text{ if } \| z_{n_k}\|_{L^2(\Omega)} < \delta \\
             &a_{n_k} y_{n_k} \oplus h^{-}(a_{n_k} y_{n_k})
                       &&\quad \text{ if } \| z_{n_k}\|_{L^2(\Omega)} \geq \delta.
            \end{aligned}
            \right .
\end{equation}
Clearly, $\|z_{n_k} - u_{n_k}\|_{L^2(\Omega)} = 0$ if 
$\| z_{n_k}\|_{L^2(\Omega)} < \delta$. Moreover, if $\| z_{n_k}\|_{L^2(\Omega)} \geq \delta$,
then
\begin{displaymath}
\begin{aligned}
 \|z_{n_k} - u_{n_k}\|_{L^2(\Omega)} 
 &= \| (y_{n_k} \oplus h^{-}(y_{n_k})) - 
   (a_{n_k} y_{n_k} \oplus h^{-}(a_{n_k} y_{n_k}))\|_{L^2(\Omega)} \\
 %&\leq \| y_{n_k} - a_{n_k} y_{n_k}\|_{L^2(\Omega)}
 %   + \|h^{-}(y_{n_k}) - h^{-}(a_{n_k} y_{n_k})\|_{L^2(\Omega)} \\
 &\leq \| y_{n_k} - a_{n_k} y_{n_k}\|_{L^2(\Omega)}
    + \|h^{-}(y_{n_k}) - h^{-}(a_{n_k} y_{n_k})\|_{X^{+}} \\
 &\leq \| y_{n_k} - a_{n_k} y_{n_k}\|_{L^2(\Omega)}
    + \mu \|y_{n_k} - a_{n_k} y_{n_k}\|_{X^{-}} \\
 %&\leq  (1+\mu M_1) \| y_{n_k} - a_{n_k} y_{n_k}\|_{L^2(\Omega)} \\
 &\leq  (1+\mu M_1) |1-a_{n_k}| \; \| y_{n_k}\|_{L^2(\Omega)} \\
 &\leq  (1+\mu M_1) \frac{b \zeta_{n_k}}{\|y_{n_k}\|_{L^2(\Omega)} + \|h^{-}(y_{n_k})\|_{L^2(\Omega)}}    
         \| y_{n_k}\|_{L^2(\Omega)} \\
 &\leq  (1+\mu M_1) b \zeta_{n_k}.
\end{aligned}
\end{displaymath} 
Hence, $\|z_{n_k} - u_{n_k}\|_{L^2(\Omega)} \leq  (1+\mu M_1) b \zeta_{n_k}$ 
for all $k \in \mathbb N$. As $\zeta_{n_k} \rightarrow 0$, we conclude that
\begin{equation}
 \|z_{n_k} - u_{n_k}\|_{L^2(\Omega)} \rightarrow 0
\end{equation}
as $k \rightarrow \infty$. It remains to show that $u_{n_k} \in B_{L^2(\Omega)}(0,\delta)$
for all $k \in \mathbb N$.  If $\| z_{n_k}\|_{L^2(\Omega)} < \delta$, then
$u_{n_k} \in B_{L^2(\Omega)}(0,\delta)$. 
If $\| z_{n_k}\|_{L^2(\Omega)} \geq \delta$, we can write
\begin{equation}
\label{eq:ApproxNormUnk}
 \begin{aligned}
  \|u_{n_k}\|_{L^2(\Omega)} 
  &\leq \|u_{n_k} - a_{n_k} z_{n_k}\|_{L^2(\Omega)} + \|a_{n_k} z_{n_k}\|_{L^2(\Omega)}\\
  &= \| (a_{n_k} y_{n_k} \oplus h^{-}(a_{n_k} y_{n_k}))
      - a_{n_k} (y_{n_k} \oplus h^{-}(y_{n_k})) \|_{L^2(\Omega)}\\
     &\quad + \|a_{n_k} z_{n_k}\|_{L^2(\Omega)}\\
  &\leq \| h^{-}(a_{n_k} y_{n_k}) - a_{n_k} h^{-}(y_{n_k}) \|_{L^2(\Omega)}
       + \|a_{n_k} z_{n_k}\|_{L^2(\Omega)}\\
  &\leq \| h^{-}(a_{n_k} y_{n_k}) - h^{-}(y_{n_k})\|_{L^2(\Omega)}
       + \|h^{-}(y_{n_k}) - a_{n_k} h^{-}(y_{n_k}) \|_{L^2(\Omega)} \\
     &\quad + \|a_{n_k} z_{n_k}\|_{L^2(\Omega)}.
 \end{aligned}
\end{equation}
Now, if $\| z_{n_k}\|_{L^2(\Omega)} \geq \delta$, then by the Lipschitz continuity of $h^{-}$ and \eqref{eq:mu} 
\begin{equation}
\label{eq:ApproxUnk1}
 \begin{aligned}
  \| h^{-}(a_{n_k} y_{n_k}) - h^{-}(y_{n_k})\|_{L^2(\Omega)}
  &\leq  \| h^{-}(a_{n_k} y_{n_k}) - h^{-}(y_{n_k})\|_{X^{+}} \\
  &\leq  \mu \| a_{n_k} y_{n_k} - y_{n_k}\|_{X^{-}} \\
  &\leq  \mu M_1 |a_{n_k} -1| \; \|y_{n_k}\|_{L^2(\Omega)} \\
  &= \frac{\mu_0}{M_1 M_2} M_1 
     \frac{b \zeta_{n_k} \|y_{n_k}\|_{L^2(\Omega)} }{\|y_{n_k}\|_{L^2(\Omega)} + \|h^{-}(y_{n_k})\|_{L^2(\Omega)}}  \\
  &\leq \mu_0 b \zeta_{n_k}. 
 \end{aligned}
\end{equation}
Similarly, if $\| z_{n_k}\|_{L^2(\Omega)} \geq \delta$, then
\begin{equation}
\label{eq:ApproxUnk2}
 \begin{aligned}
 \|h^{-}(y_{n_k}) - a_{n_k} h^{-}(y_{n_k}) \|_{L^2(\Omega)}
 %&\leq |1-a_{n_k}| \; \| h^{-}(y_{n_k}) \|_{L^2(\Omega)} \\
 &\leq |1-a_{n_k}| \; \| h^{-}(y_{n_k}) \|_{X^{+}} \\
 &\leq \mu |1-a_{n_k}| \; \|y_{n_k} \|_{X^{-}} \\
 &\leq \mu M_1|1-a_{n_k}| \;  \| y_{n_k} \|_{L^2(\Omega)} \\
 &= \frac{\mu_0}{M_1 M_2} M_1 
     \frac{b \zeta_{n_k} \|y_{n_k}\|_{L^2(\Omega)} }{\|y_{n_k}\|_{L^2(\Omega)} + \|h^{-}(y_{n_k})\|_{L^2(\Omega)}}  \\
 &\leq \mu_0 b \zeta_{n_k}. 
 \end{aligned}
 \end{equation}
 Since $\|z_{n_k}\|_{L^2(\Omega)} \geq (\|y_{n_k}\|_{L^2(\Omega)} + \|h^{-}(y_{n_k})\|_{L^2(\Omega)})
 /(\|P^{+}\| + \|P^{-}\|)$, it follows that
 \begin{displaymath}
  \frac{b \zeta_{n_k} \|z_{n_k}\|_{L^2(\Omega)}}{\|y_{n_k}\|_{L^2(\Omega)} + \|h^{-}(y_{n_k})\|_{L^2(\Omega)}}
  \geq \frac{b \zeta_{n_k}}{\|P^{+}\| + \|P^{-}\|}.
 \end{displaymath}
 Hence, if $\| z_{n_k}\|_{L^2(\Omega)} \geq \delta$, then
\begin{equation}
\label{eq:ApproxUnk3}  
\begin{aligned} 
   \|a_{n_k} z_{n_k}\|_{L^2(\Omega)} 
  % &= |a_{n_k}| \|z_{n_k}\|_{L^2(\Omega)} \\
  % &= a_{n_k} \|z_{n_k}\|_{L^2(\Omega)} \\
   &= \left ( 1 - \frac{b \zeta_{n_k}}{\|y_{n_k}\|_{L^2(\Omega)} + \|h^{-}(y_{n_k})\|_{L^2(\Omega)}} \right )
      \|z_{n_k}\|_{L^2(\Omega)} \\
   &= \|z_{n_k}\|_{L^2(\Omega)} 
      - \frac{b \zeta_{n_k} \|z_{n_k}\|_{L^2(\Omega)}}{\|y_{n_k}\|_{L^2(\Omega)} + \|h^{-}(y_{n_k})\|_{L^2(\Omega)}} \\
   &\leq  \|z_{n_k}\|_{L^2(\Omega)}   - \frac{b \zeta_{n_k}}{\|P^{+}\|+\|P^{-}\|} \\
   &< \delta + \zeta_{n_k} - \frac{b \zeta_{n_k}}{\|P^{+}\|+\|P^{-}\|}.
\end{aligned}
\end{equation}
Therefore, by \eqref{eq:ApproxNormUnk} -- \eqref{eq:ApproxUnk3},  
if $\| z_{n_k}\|_{L^2(\Omega)} \geq \delta$, then
\begin{equation}
\label{eq:ApproxNormUnk2}
\begin{aligned}
  \|u_{n_k}\|_{L^2(\Omega)}  
   &< \mu_0 b \zeta_{n_k} + \mu_0 b \zeta_{n_k} 
    + \delta + \zeta_{n_k} - \frac{b \zeta_{n_k}}{\|P^{+}\|+\|P^{-}\|} \\
   &= \delta +  \Big ( 2\mu_0 b - \frac{b}{\|P^{+}\|+\|P^{-}\|}  +1 \Big ) \zeta_{n_k}.
\end{aligned} 
\end{equation}
By the choice of $b$ in \eqref{eq:constB}, we get
\begin{displaymath}
 \begin{aligned}
 2\mu_0 b - \frac{b}{\|P^{+}\|+\|P^{-}\|} +1
  %&= \left ( 2\mu_0  - \frac{1}{\|P^{+}\|+\|P^{-}\|} \right )b +1 \\
  &= - \Big ((\|P^{+}\|+\|P^{-}\|)^{-1} - 2 \mu_0 \Big )b +1\\
  &< -1 +1 \\
  &= 0.
 \end{aligned}
\end{displaymath}
%Therefore $2\mu_0 b - \frac{b}{\|P^{+}\|+\|P^{-}\|} +1 < 0$.
It follows from \eqref{eq:ApproxNormUnk2} that $\|u_{n_k}\|_{L^2(\Omega)}  < \delta$ 
if $\| z_{n_k}\|_{L^2(\Omega)} \geq \delta$. Hence, we conclude that
$u_{n_k} \in \text{\upshape{graph}}(h^{-}) \cap B_{L^2(\Omega)}(0, \delta)$
for all $k \in \mathbb N$ and statement (i) follows.

Statement (ii) can be proved similarly. The only difference is that the sequence $z_n$
belongs to different spaces $L^2(\Omega_n)$ for each $n \in \mathbb
N$. We only need to adjust the proof in part (i) and keep track of the
dependence on $n$. In particular,  we replace \eqref{eq:constB} by
\begin{displaymath}
  b > \sup_{n \in \mathbb N} \left \{ \frac{1}{(\|P_n^{+}\| + \|P_n^{-}\|)^{-1} -2 \mu_0} \right \} >0
\end{displaymath}  
and \eqref{eq:constAnk} by
\begin{displaymath}
 a_{n_k} := 1 - \frac{b \zeta_{n_k}}{\|y_{n_k}\|_{L^2(\Omega_{n_k})} + \|h_{n_k}^{-}(y_{n_k})\|_{L^2(\Omega_{n_k})}}, 
\end{displaymath}
for each $k \in \mathbb N$.
\end{proof}
We now show the upper semicontinuity of local stable invariant manifolds.

\begin{proof}[Proof of Theorem \ref{th:mainResults2} (i)]
 By Lemma \ref{lem:upperSemiContCharac}, we need to show that 
 for any sequence $\{\xi_n\}_{n \in \mathbb N}$ with 
 $\xi_n \in \text{\upshape{graph}}(h_n^{-})\cap B_n$, if $\{\xi_{n_k} \}_{k \in \mathbb N}$ is
 a subsequence then there exist a further subsequence (denoted again by $\xi_{n_k}$)
 and  a sequence $\{u_{n_k}\}_{k \in \mathbb N}$ with $u_{n_k} \in \text{\upshape{graph}}(h^{-})\cap B$ 
 such that $\| \xi_{n_k} -u_{n_k} \|_{L^2(D)} \rightarrow 0$  as $k \rightarrow
 \infty$.  

 Let $\{\xi_n\}_{n \in \mathbb N}$ be a sequence with $\xi_n \in
 \text{\upshape{graph}}(h_n^{-})\cap B_n$ and $(\xi_{n_k})_{k \in \mathbb N}$ 
 be an arbitrary subsequence. We write $\xi_{n_k} := v_{n_k} \oplus
 h_{n_k}^{-}(v_{n_k})$ for some $v_{n_k} \in X_{n_k}^{-}$. Since
 $\|\xi_{n_k}\|_{L^2(\Omega_{n_k})} = \| v_{n_k} \oplus h_{n_k}^{-}(v_{n_k}) \|_{L^2(\Omega_{n_k})} < \delta$
 for all $k \in \mathbb N$, we can extract a subsequence of $v_{n_k}$
 (indexed again by $n_k$) such that 
\begin{equation} 
\label{eq:weakConvVnk}
     v_{n_k} \rightharpoonup v
\end{equation}
in $L^2(D)$ as $k \rightarrow \infty$. By the assumption that $|\Omega_n|
\rightarrow |\Omega|$, we conclude that $v =0$ almost everywhere in $D \backslash \Omega$, that is,
$v \in L^2(\Omega)$. 
% Hence,
% \begin{equation}
% \label{eq:weakConvVnk}
%  i_{n_k} (v_{n_k}) \rightharpoonup i(v)
% \end{equation}
% in $L^2(D)$ as $k \rightarrow \infty$. 
Moreover, by the convergence of
$P_n^{-} \rightarrow P^{-}$ in $\mathscr L(L^2(D))$ (see Remark \ref{remark:NonConvPminus}) and
%spectral projections \eqref{eq:convProminus} and 
the weak convergence of $v_{n_k}$, 
it is easy to see that
% it follows that
% \begin{displaymath}
%  \begin{aligned}
% \left |\left ( v_{n_k} -  P^{-} v \;  |\;  \phi \right )_{L^2(D)}\right |
% & \leq \left | \left ( v_{n_k} -  P^{-} v_{n_k}|_{\Omega} \;|\; \phi \right )_{L^2(D)} \right |\\
% &\quad  + \left |\left (P^{-} v_{n_k}|_{\Omega} - P^{-} v \;|\; \phi \right )_{L^2(D)} \right |\\
% &\leq \| P_{n_k}^{-}  -  P^{-} \| \; \| v_{n_k}\|_{L^2(D)} \; \|\phi\|_{L^2(D)}\\
% &\quad + \left |\left ( v_{n_k} - v \;|\; (P^{-})^{*} \phi \right)_{L^2(D)} \right | \\ 
% &\rightarrow 0 
% \end{aligned}
% \end{displaymath}
% as $k \rightarrow \infty$ for all $\phi \in L^2(D)$, where $(P^{-})^{*}$
% is the adjoint operator of  $P^{-}$. This means 
$v_{n_k} \rightharpoonup P^{-} v$ in $L^2(D)$ as $k
\rightarrow \infty$. By the uniqueness of weak limit, $v = P^{-} v$
and hence $v \in X^{-}$.
Since $\|h_{n_k}^{-}(v_{n_k})\|_{L^2(D)}$ is uniformly bounded, we
can apply Corollary \ref{cor:convBoundSeqInXplus} to extract a further
subsequence (indexed again by $n_k$) such that 
\begin{equation}
\label{eq:strongConvHVnk}
     h_{n_k}^{-}(v_{n_k}) \rightarrow  w
\end{equation}
in $L^2(D)$ as $k \rightarrow \infty$ with the limit $w \in X^{+}$.
Thus, we get
\begin{equation}
 \label{eq:convStabVW}
    v_{n_k} \oplus h_{n_k}^{-}(v_{n_k}) \rightharpoonup  v \oplus w
\end{equation}
in $L^2(D)$ as $k \rightarrow \infty$. By a standard property of weak convergence,
\begin{equation}
 \label{eq:weakNormVW}
 \begin{aligned}
 \| v \oplus w \|_{L^2(D)} 
  \leq \liminf_{k \rightarrow \infty} \| v_{n_k} \oplus h_{n_k}^{-}(v_{n_k}) \|_{L^2(D)} %\\
  %& \leq \limsup_{k \rightarrow \infty}\| v_{n_k} \oplus h_{n_k}^{-}(v_{n_k})\|_{L^2(D)}\\
  \leq \delta.
\end{aligned}
\end{equation}
Hence, $u:= v \oplus w$ belongs to $\overline{B}$. Applying \eqref{eq:convSolUsingSemiflow}, 
we get from \eqref{eq:convStabVW} and globally Lipschitz assumption for the modified function 
$\tilde f$ that
%\begin{displaymath}
    $\Phi_{t,n_k} (v_{n_k} \oplus h_{n_k}^{-}(v_{n_k}))
\rightarrow   \Phi_{t} (v \oplus w) $
%\end{displaymath} 
in $L^2(D)$ as $k \rightarrow \infty$ for all $t > 0$.
%By convergence of spectral projections \eqref{eq:convProminus}, \eqref{eq:convProplus},
Lemma \ref{lem:convUsUcUu}  implies that
\begin{displaymath}
\begin{aligned}
     P_{n_k}^{-} \Phi_{t,n_k} (v_{n_k} \oplus h_{n_k}^{-}(v_{n_k})), 
   &\rightarrow   P^{-} \Phi_{t} (v \oplus w) \\
% \end{displaymath}
% and 
% \begin{displaymath}
    P_{n_k}^{+} \Phi_{t,n_k} (v_{n_k} \oplus h_{n_k}^{-}(v_{n_k})) 
   &\rightarrow  P^{+} \Phi_{t} (v \oplus w) 
\end{aligned}
\end{displaymath}
in $L^2(D)$ as $k \rightarrow \infty$ for all $t > 0$.
By the construction of  $h_{n_k}^{-}(v_{n_k})$ 
(see Theorem \ref{th:globalStaMan}), we have that
\begin{displaymath}
  \| P_{n_k}^{+} \Phi_{t,n_k} (v_{n_k} \oplus h_{n_k}^{-}(v_{n_k})) \|_{X_{n_k}^{+}}
  \leq \mu \|P_{n_k}^{-} \Phi_{t,n_k} (v_{n_k} \oplus h_{n_k}^{-}(v_{n_k}))\|_{X_{n_k}^{-}},
\end{displaymath}
for all $t \geq 0$.  
% Taking the norms on $X_n^{+}$ and $X_n^{-}$ in
% \eqref{eq:equivNormXminusN} and \eqref{eq:equivNormXplusN} into account, 
The above implies
\begin{displaymath}
  \| P_{n_k}^{+} \Phi_{t,n_k} (v_{n_k} \oplus h_{n_k}^{-}(v_{n_k})) \|_{L^2(\Omega_{n_k})}
  \leq \mu M_1 \|P_{n_k}^{-} \Phi_{t,n_k} (v_{n_k} \oplus h_{n_k}^{-}(v_{n_k}))\|_{L^2(\Omega_{n_k})},
\end{displaymath}
for all $t \geq 0$. Passing to the limit as $k \rightarrow \infty$, we
obtain
\begin{displaymath}
  \| P^{+} \Phi_{t} (v \oplus w) \|_{L^2(\Omega)}
  \leq \mu M_1 \|P^{-} \Phi_{t} (v \oplus w)\|_{L^2(\Omega)}
\end{displaymath}
for all $t > 0$. By the assumptions on $\mu_0$ and $\mu$ in \eqref{eq:mu0} and \eqref{eq:mu},
and the equivalence of norms on $X^{-}$ and $X^{+}$, 
%(see \eqref{eq:equivNormXminus} and \eqref{eq:equivNormXplus}),
it follows that
\begin{equation} 
 \label{eq:coneCondPosT}
  \| P^{+} \Phi_{t} (v \oplus w) \|_{X^{+}}
  \leq \mu M_1 M_2 \|P^{-} \Phi_{t} (v \oplus w)\|_{X^{-}}
   = \mu_0 \|P^{-} \Phi_{t} (v \oplus w)\|_{X^{-}},
\end{equation}
for all $t > 0$. We claim that $\|w\|_{X^{+}} \leq \mu_0 \|v\|_{X^{-}}$.
If $\|w\|_{X^{+}} > \mu_0 \|v\|_{X^{-}}$, that is $v \oplus w$ is in the
interior of the cone $K_{\mu_0}$ defined by \eqref{eq:conKlambda}, 
we can find a product neighbourhood
$U(v,w)$ of $v \oplus w$ such that $U(v,w) \subset
\text{Int}(K_{\mu_0})$. Since the solution of parabolic equation with
the initial condition $v \oplus w$ is continuous, there exists $t_0 > 0$
such that $\Phi_t(v \oplus w) \in U(v,w)$ for $0 \leq t \leq t_0$. This
implies that  
$ \| P^{+} \Phi_{t} (v \oplus w) \|_{X^{+}} > \mu_0 \|P^{-} \Phi_{t} (v
\oplus w)\|_{X^{-}}$ 
for $0 \leq t \leq  t_0$, which is a contradiction to
\eqref{eq:coneCondPosT}.  Hence, by the definition of $\hat{h}^{-}$ (a
modification with the cone $K_{\mu_0}$), we conclude that $w =
\hat{h}^{-}(v)$. As both modification agree on $\bar{B}$, we have $w = h^{-}(v)$.
Therefore, \eqref{eq:strongConvHVnk} implies
\begin{equation}
\label{eq:strongConvToHv}
    h_{n_k}^{-}(v_{n_k}) \rightarrow  h^{-}(v)
\end{equation}
in $L^2(D)$ as $k \rightarrow \infty$.

The remainder of this proof deals with the existence of the required sequence
$u_{n_k} \in \text{graph}(h^{-}) \cap B$. At this stage, we keep the
index of our subsequence as in the previous part. We define
%\begin{displaymath}
 $y_{n_k} := P^{-} v_{n_k}|_{\Omega} \in X^{-}$
%\end{displaymath}
for each $k \in \mathbb N$.
%By convergence of spectral projections \eqref{eq:convProminus} 
By the convergence 
$P_n^{-} \rightarrow P^{-}$  in $\mathscr L(L^2(D))$ (from Remark \ref{remark:NonConvPminus})
and the boundedness of $\|v_{n_k}\|_{L^2(D)}$, we get 
\begin{equation}
\label{eq:ynkBounded}
\begin{aligned}
\| y_{n_k} - v_{n_k} \|_{L^2(D)}
%& =\| P^{-} v_{n_k}|_{\Omega} -  v_{n_k} \|_{L^2(D)}\\
 \leq \| P^{-}  - P_{n_k}^{-} \| \| v_{n_k}\|_{L^2(D)} %\\
 \rightarrow 0
\end{aligned}
\end{equation} 
as $k \rightarrow \infty$. In particular, $\|y_{n_k}\|_{L^2(\Omega)}$ is
uniformly bounded.  Moreover, by \eqref{eq:ynkBounded} and \eqref{eq:weakConvVnk}, we get 
%the weak convergence of $v_{n_k}$,
% \begin{displaymath}
%  \begin{aligned}
%   \left | \left ( y_{n_k} - v \;|\; \phi \right )_{L^2(D)} \right |
%   &\leq \left | \left ( y_{n_k} - v_{n_k} \;|\; \phi \right )_{L^2(D)} \right |\\
%   &\quad + \left | \left ( v_{n_k} - v  \;|\; \phi \right )_{L^2(D)} \right |\\
%   &\leq \| y_{n_k} -  v_{n_k})\|_{L^2(D)} \; \|\phi\|_{L^2(D)}\\
%   &\quad +  \left | \left ( v_{n_k} - v \;|\; \phi \right )_{L^2(D)} \right |\\
%   &\rightarrow 0
% \end{aligned}
% \end{displaymath}
% as $k \rightarrow \infty$ for all $\phi \in L^2(D)$. Hence, 
\begin{equation}
 \label{eq:weakConvYnk}
      y_{n_k} \rightharpoonup v
\end{equation}
 in $L^2(D)$ as $k \rightarrow \infty$. 
By the Lipschitz continuity of $h^{-}$, %together with \eqref{eq:equivNormXminus} and \eqref{eq:equivNormXplus},
% \begin{displaymath}
%    \| h^{-}(y_{n_k}) \|_{L^2(\Omega)} 
%    \leq  \| h^{-}(y_{n_k}) \|_{X^{+}}
%    \leq \mu \| y_{n_k} \|_{X^{-}}
%    \leq \mu M_1  \|y_{n_k} \|_{L^2(\Omega)}
%    < \infty                   
% \end{displaymath}
% for all $k \in \mathbb N$. 
 $\| h^{-}(y_{n_k}) \|_{L^2(\Omega)}$ is uniformly bounded.
Since $X^{+}$ is a finite dimensional space,
we can extract a further subsequence (indexed again by $n_k$) such that 
\begin{equation}
 \label{eq:strongConvHynk}
   h^{-}(y_{n_k}) \rightarrow  \tilde w
\end{equation}
in $L^2(D)$ as $k \rightarrow \infty$ with the limit $\tilde w \in
X^{+}$. Therefore,
% \begin{equation}
%  \label{eq:convStabVtildeW}
   $y_{n_k} \oplus h^{-}(y_{n_k}) \rightharpoonup  v \oplus \tilde w$
%\end{equation}
in $L^2(D)$ as $k \rightarrow \infty$.
By %the convergence of solutions under domain perturbation in
\eqref{eq:convSolUsingSemiflow} (with $\Omega_n = \Omega$
for all $n \in \mathbb N$), %we get from \eqref{eq:convStabVtildeW}  that
it follows that
%\begin{displaymath}
   $\Phi_{t} (y_{n_k} \oplus h^{-}(y_{n_k})) \rightarrow  \Phi_{t} (v \oplus \tilde w)$ 
%\end{displaymath} 
in $L^2(D)$ as $k \rightarrow \infty$ for all $t > 0$.
Hence,
\begin{displaymath}
\begin{aligned}
    P^{-} \Phi_{t} (y_{n_k} \oplus h^{-}(y_{n_k})) 
   &\rightarrow   P^{-} \Phi_{t} (v \oplus \tilde w), \\ 
% \end{displaymath}
% and  
% \begin{displaymath}
    P^{+} \Phi_{t} (y_{n_k} \oplus h^{-}(y_{n_k})) 
   &\rightarrow   P^{+} \Phi_{t} (v \oplus \tilde w) 
\end{aligned}
\end{displaymath}
in $L^2(\Omega)$ as $k \rightarrow \infty$ for all $t > 0$.
Since these sequences are in the fixed spaces $X^{-}$ and $X^{+}$
respectively, \eqref{eq:equivNormXminusplus}
%and \eqref{eq:equivNormXplus} 
implies that they
converge under $\|\cdot\|_{X^{-}}$ and $\| \cdot \|_{X^{+}}$, respectively.
%
%    
%
%
%Hence
%\begin{displaymath}
%   i P^{+} \Phi_{t} (y_{n_k} \oplus h^{-}(y_{n_k})) 
%   \rightarrow  i P^{+} \Phi_{t} (v \oplus \tilde w) 
%\end{displaymath}
%and 
%\begin{displaymath}
%   i P^{-} \Phi_{t} (y_{n_k} \oplus h^{-}(y_{n_k})) 
%   \rightarrow  i P^{-} \Phi_{t} (v \oplus \tilde w) 
%\end{displaymath}
%in $L^2(D)$ as $k \rightarrow \infty$ for all $t > 0$.
%
%
%
By the construction of  $h^{-}(y_{n_k})$ (see Theorem \ref{th:globalStaMan}), we have that
\begin{displaymath}
  \| P^{+} \Phi_{t} (y_{n_k} \oplus h^{-}(y_{n_k})) \|_{X^{+}}
  \leq \mu \|P^{-} \Phi_{t} (y_{n_k} \oplus h^{-}(y_{n_k}))\|_{X^{-}},
\end{displaymath}
for all $t \geq 0$. 
%By equivalence of norms on $X^{+}$ and $X^{-}$, 
%\begin{displaymath}
%  \| P^{+} \Phi_{t} (y_{n_k} \oplus h^{-}(y_{n_k})) \|_{L^2(\Omega)}
%  \leq \mu M_1 \|P^{-} \Phi_{t} (y_{n_k} \oplus h^{-}(y_{n_k}))\|_{L^2(\Omega)}
%\end{displaymath}
%for all $t \geq 0$. 
Passing to the limit as $k \rightarrow \infty$, we obtain
%\begin{displaymath}
%  \| P^{+} \Phi_{t} (v \oplus \tilde w) \|_{L^2(\Omega)}
%  \leq \mu M_1 \|P^{-} \Phi_{t} (v \oplus \tilde w\|_{L^2(\Omega)}
%\end{displaymath}
%for all $t > 0$. 
%By equivalence of norms on $X^{-}$ and $X^{+}$, and the
%assumption on $\mu_0, \mu$ in \eqref{eq:mu0} and \eqref{eq:mu}, it
%follows that
\begin{equation}
 \label{eq:coneCondPosTWtilde}
  \| P^{+} \Phi_{t} (v \oplus \tilde w) \|_{X^{+}}
 % \leq \mu M_1 M_2 \|P^{-} \Phi_{t} (v \oplus \tilde w)\|_{X^{-}}
   = \mu \|P^{-} \Phi_{t} (v \oplus \tilde w)\|_{X^{-}}
\end{equation}
for all $t > 0$. 
By a similar argument appeared after \eqref{eq:coneCondPosT}, 
we conclude that $\|\tilde w\|_{X^{+}} \leq \mu \|v\|_{X^{-}}$.
% We claim that $\|\tilde w\|_{X^{+}} \leq \mu_0 \|v\|_{X^{-}}$.
% If $\|\tilde w\|_{X^{+}} > \mu_0 \|v\|_{X^{-}}$, that is $v \oplus \tilde w$ is in the
% interior of the cone $K_{\mu_0}$, we can find a product neighbourhood
% $U(v,\tilde w)$ of $v \oplus \tilde w$ such that $U(v,\tilde w) \subset
% \text{Int}(K_{\mu_0})$. Since the solution of parabolic equation with
% the initial condition $v \oplus \tilde w$ is continuous, there exists $t_0 > 0$
% such that $\Phi_t(v \oplus \tilde w) \in U(v,\tilde w)$ for $0 \leq t \leq t_0$. This
% implies that  
% $ \| P^{+} \Phi_{t} (v \oplus \tilde w) \|_{X^{+}} > \mu_0 \|P^{-} \Phi_{t} (v
% \oplus \tilde w)\|_{X^{-}}$ 
% for $0 \leq t \leq  t_0$, which is a contradiction to
% \eqref{eq:coneCondPosT}.  
%Hence, by definition of $\hat{h}^{-}$ (a modification with the cone $K_{\mu_0}$)
% we also have $\tilde w = w = \hat{h}^{-}(v) = h^{-}(v)$.
Hence, $\tilde w$ agrees with $w = h^{-}(v)$.
Therefore, \eqref{eq:strongConvHynk} implies
\begin{equation}
\label{eq:strongConvToHynk}
   h^{-}(y_{n_k}) \rightarrow  h^{-}(v)
\end{equation}
in $L^2(D)$ as $k \rightarrow \infty$. 
Recall that $\xi_{n_k} = v_{n_k} \oplus h_{n_k}^{-}(v_{n_k})$.
If we set $z_{n_k} := y_{n_k} \oplus h^{-}(y_{n_k}) \in \text{graph}(h^{-})$,
then by \eqref{eq:strongConvToHv}, \eqref{eq:ynkBounded} and \eqref{eq:strongConvToHynk}, we get
\begin{equation}
 \label{eq:convXiZ}
 \begin{aligned}
 &\| \xi_{n_k} - z_{n_k} \|_{L^2(D)} \\
 &\quad = \| (v_{n_k} \oplus h_{n_k}^{-}(v_{n_k})) 
          - (y_{n_k} \oplus h^{-}(y_{n_k})) \|_{L^2(D)} \\
 &\quad \leq \| v_{n_k} - y_{n_k} \|_{L^2(D)} 
          + \|  h_{n_k}^{-}(v_{n_k}) -  h^{-}(y_{n_k})\|_{L^2(D)} \\
 &\quad \leq \| v_{n_k} -  y_{n_k} \|_{L^2(D)} 
          + \| h_{n_k}^{-}(v_{n_k}) -  h^{-}(v)\|_{L^2(D)} \\
 &\quad \quad  + \| h^{-}(v) -  h^{-}(y_{n_k})\|_{L^2(D)} \\
 &\quad \rightarrow 0
  \end{aligned}
\end{equation}
as $k \rightarrow \infty$. Therefore, we can extract a further subsequence
(indexed again by $n_k$) and $\zeta_{n_k} > 0$ with $\zeta_{n_k} \rightarrow 0$ 
as $k \rightarrow \infty$ such that
%\begin{displaymath}
  $\|  \xi_{n_k} - z_{n_k} \|_{L^2(D)} < \zeta_{n_k}$
%\end{displaymath}
for all $k \in \mathbb N$. It follows that
\begin{displaymath}
 \| z_{n_k} \|_{L^2(\Omega)} \leq \| \xi_{n_k} \|_{L^2(\Omega_{n_k})} + \zeta_{n_k}
                          < \delta + \zeta_{n_k},
\end{displaymath}
for all $k \in \mathbb N$, that is, $z_{n_k} \in \text{graph}(h^{-}) 
\cap B_{L^2(\Omega)}(0, \delta+ \zeta_{n_k})$ for all $k \in \mathbb N$.
We can apply Lemma \ref{lem:seqOfGphSmallerBall} (i) to obtain a subsequence 
(indexed again by $n_k$) $z_{n_k}$ and a sequence $u_{n_k} \in \text{graph}(h^{-}) \cap B$
such that $\| z_{n_k} - u_{n_k}\|_{L^2(\Omega)} \rightarrow 0$ as 
$k \rightarrow \infty$.
It follows from \eqref{eq:convXiZ} that
\begin{displaymath}
%\begin{aligned}
 \| \xi_{n_k} -  u_{n_k} \|_{L^2(D)}
 \leq \| \xi_{n_k} - z_{n_k} \|_{L^2(D)} + \| z_{n_k} - u_{n_k} \|_{L^2(D)} 
 \rightarrow 0
 %\end{aligned}
\end{displaymath}
as $k \rightarrow \infty$. 
Hence, we obtain the required sequence $u_{n_k}$.
Since we start with an arbitrary sequence $\xi_n \in \text{\upshape{graph}}(h_n^{-})\cap B_n$,
the assertion of Theorem \ref{th:mainResults2} (i) follows.  
\end{proof}
The lower semicontinuity of local stable invariant manifolds can be obtained
by a similar fashion.
\begin{proof}[Proof of Theorem \ref{th:mainResults2} (ii)]
The statement follows by a similar argument to the proof of Theorem \ref{th:mainResults2} (i).
We use Lemma \ref{lem:lowerSemiContCharac} and Lemma
\ref{lem:seqOfGphSmallerBall} (ii) instead of Lemma
\ref{lem:upperSemiContCharac} and \ref{lem:seqOfGphSmallerBall} (i). 
\end{proof}

%%%%%%%%%%%%%%%%%%%%%%%%%%%%%%%%%%%%%%%%%%%%%%%%%%%%%%%%%%%%%%%%%%%%%%%%%%%%%%%%%%%%%%%%%%%%%%%%%%% 
\subsection*{Acknowledgment}
The author would like to thank  D. Daners and E. N. Dancer for helpful discussions
and suggestions.

%%%%%%%%%%%%%%%%%%%%%%%%%%%%%%%%%%%%%%%%%%%%%%%%%%%%%%%%%%%%%%%%%%%%%%%%%%%%%%%%%%%%%%%%%%%%%%%%%%%

%\begin{thebibliography}{1}
%\bibitem{daners}Daners, D., \emph{Domain perturbations for linear and
% nonlinear pararbolic equations},
%  J. Diff Eq., 1996.

%\end{thebibliography}
\bibliography{database}{}

\begin{thebibliography}{10}

\bibitem{MR1066204}
J.~Appell and P.~P. Zabrejko.
\newblock {\em Nonlinear superposition operators}, volume~95 of {\em Cambridge
  Tracts in Mathematics}.
\newblock Cambridge University Press, Cambridge, 1990.

\bibitem{MR1832168}
W.~Arendt.
\newblock Approximation of degenerate semigroups.
\newblock {\em Taiwanese J. Math.}, 5(2):279--295, 2001.

\bibitem{MR2041515}
J.~M. Arrieta and A.~N. Carvalho.
\newblock Spectral convergence and nonlinear dynamics of reaction-diffusion
  equations under perturbations of the domain.
\newblock {\em J. Differential Equations}, 199(1):143--178, 2004.

\bibitem{MR1000974}
P.~W. Bates and C.~K. R.~T. Jones.
\newblock Invariant manifolds for semilinear partial differential equations.
\newblock In {\em Dynamics reported, {V}ol.\ 2}, volume~2 of {\em Dynam.
  Report. Ser. Dynam. Systems Appl.}, pages 1--38. Wiley, Chichester, 1989.

\bibitem{MR1404388}
D.~Daners.
\newblock Domain perturbation for linear and nonlinear parabolic equations.
\newblock {\em J. Differential Equations}, 129(2):358--402, 1996.

\bibitem{MR1955096}
D.~Daners.
\newblock Dirichlet problems on varying domains.
\newblock {\em J. Differential Equations}, 188(2):591--624, 2003.

\bibitem{MR2119988}
D.~Daners.
\newblock Perturbation of semi-linear evolution equations under weak
  assumptions at initial time.
\newblock {\em J. Differential Equations}, 210(2):352--382, 2005.

\bibitem{MR1156075}
R.~Dautray and J.-L. Lions.
\newblock {\em Mathematical analysis and numerical methods for science and
  technology. {V}ol. 5}.
\newblock Springer-Verlag, Berlin, 1992.
% \newblock Evolution problems. I, With the collaboration of Michel Artola,
%   Michel Cessenat and H\'el\`ene Lanchon, Translated from the French by Alan
%   Craig.

\bibitem{MR2130211}
E.~A.~M. de~Abreu and A.~N. Carvalho.
\newblock Attractors for semilinear parabolic problems with {D}irichlet
  boundary conditions in varying domains.
\newblock {\em Mat. Contemp.}, 27:37--51, 2004.

\bibitem{MR1508998}
J.~Hadamard.
\newblock Sur l'\'equilibre des plaques \'elastiques circulaires libres ou
  appuy\'ees et celui de la sph\`ere isotrope.
\newblock {\em Ann. Sci. \'Ecole Norm. Sup. (3)}, 18:313--342, 1901.

\bibitem{MR610244}
D.~Henry.
\newblock {\em Geometric theory of semilinear parabolic equations}, volume 840
  of {\em Lecture Notes in Mathematics}.
\newblock Springer-Verlag, Berlin, 1981.

\bibitem{MR0407617}
T.~Kato.
\newblock {\em Perturbation theory for linear operators}.
\newblock Springer-Verlag, Berlin, second edition, 1976.
% \newblock Grundlehren der Mathematischen Wissenschaften, Band 132.

\bibitem{MR0298508}
U.~Mosco.
\newblock Convergence of convex sets and of solutions of variational
  inequalities.
\newblock {\em Advances in Math.}, 3:510--585, 1969.

\bibitem{MR710486}
A.~Pazy.
\newblock {\em Semigroups of linear operators and applications to partial
  differential equations}, volume~44 of {\em Applied Mathematical Sciences}.
\newblock Springer-Verlag, New York, 1983.

\bibitem{MR1834117}
M.~Prizzi and K.~P. Rybakowski.
\newblock The effect of domain squeezing upon the dynamics of
  reaction-diffusion equations.
\newblock {\em J. Differential Equations}, 173(2):271--320, 2001.

\bibitem{MR0467717}
D.~R. Smart.
\newblock {\em Fixed point theorems}.
\newblock Cambridge University Press, London, 1974.
%\newblock Cambridge Tracts in Mathematics, No. 66.

\end{thebibliography}
\bibliographystyle{abbrv}

%\email{\emph{E-mail address:} pasa4391@uni.sydney.edu.au}
\end{document}